\documentclass[10pt]{article}

\usepackage{amssymb,amsthm,amsmath,hyperref}

\numberwithin{equation}{section}
\usepackage{cases}
  \usepackage{paralist}
  \usepackage{graphics} %% add this and next lines if pictures should be in esp format
  \usepackage{epsfig} %For pictures: screened artwork should be set up with an 85 or 100 line screen
\usepackage{graphicx}  \usepackage{epstopdf}%This is to transfer .eps figure to .pdf figure; please compile your paper using PDFLeTex or PDFTeXify.
\numberwithin{equation}{section}

\newcommand\R{{\mathbb R}}
\newtheorem{claim}{Claim}[section]
\newtheorem{theorem}{Theorem}[section]
\newtheorem{corollary}[theorem]{Corollary}

\newtheorem{lemma}[theorem]{Lemma}
\newtheorem{proposition}[theorem]{Proposition}

\theoremstyle{definition}
\newtheorem{definition}[theorem]{Definition}
\newtheorem{remark}[theorem]{Remark}

\newcommand{\ep}{\varepsilon}

\newcommand{\fc}{\frac}

\allowdisplaybreaks

\begin{document}
\title{Bilinear Strichartz's type estimates in Besov spaces with application  to  inhomogeneous  nonlinear biharmonic  Schr\"odinger equation}

% Place all authors' names in [ ] shown as running head, Leave { } empty
% Please use `and' to connect the last two names if applicable
% Use FirstNameInitial.  MiddleNameInitial. LastName, or only last names of authors if there are too many authors
\author{Xuan Liu, Ting Zhang\\
	\small{School of Mathematical Sciences, Zhejiang University,
		Hangzhou 310027, China}}
\date{}

\maketitle

% The name of the associate editor will be entered by an editorial staff
% "Communicated by the associate editor name" is not needed for special issue.

%The abstract of your paper
\begin{abstract}
	In this paper, we consider the well-posedness of  the  inhomogeneous nonlinear biharmonic  Schr\"odinger equation
	with spatial inhomogeneity coefficient $K(x)$ behaves like $\left|x\right|^{-b}$ for $0<b<\min \left\{\frac{N}{2},4\right\} $. We show the local well-posedness in the whole $H^s$-subcritical case, with $0<s\le2$.  The difficulties of this problem come from the singularity of $K(x)$ and the lack of differentiability of the nonlinear term.  To resolve this, we derive  the bilinear Strichartz's type estimates for the nonlinear biharmonic Schr\"odinger equations in  Besov spaces.
\end{abstract}
\section{Introduction}
In this paper, we consider the Cauchy problem associated to the inhomogeneous biharmonic nonlinear Schr\"{o}dinger equation
\begin{equation}\label{NLS}
	\left\{\begin{array}{l}
		i \partial_{t} u+\Delta^{2} u+\mu\Delta u+K(x)f(u)=0, \quad t \in \mathbb{R}, x \in \mathbb{R}^{N} \\
		u(0, x)=\varphi(x)
	\end{array}\right.
\end{equation}
where $N\ge1$,  $\mu=-1$ or $0$, $u:\mathbb{R}\times\mathbb{R}^N\rightarrow\mathbb{C}$ is a complex-valued function, $K(x)$, $f(u)$ are  the  inhomogeneity coefficient and the nonlinear term,  respectively.  Note that if $\mu=0$ and $K(x)=\lambda \left|x\right|^{-b}$, $f(u)= \left|u\right|^{\alpha}u$ with $\lambda\in\mathbb{C}$, $0<b<\min \left\{4,\frac{N}{2}\right\} $,  $\alpha>0$,  the equation (\ref{NLS})  is invariant under the scaling, $
u_k(t,x)=k^{\frac{4-b}\alpha}u(k^4t,kx), k>0.
$  This means if $u$ is a solution of (\ref{NLS}) with the  initial datum $\phi$,  so is $u_k$ with the initial datum $\phi_k=k^{\frac{4-b}\alpha}\phi(kx)$.  Computing the homogeneous Sobolev norm, we get
$$
\left\|\phi_k\right\|_{\dot{H}^s}=k^{s-\frac{N}{2}+\frac{4-b}{\alpha}}\left\|\phi\right\|_{\dot{H}^s}.
$$
The Sobolev index which leaves the scaling symmetry invariant is called the critical index and is defined as  $s_c=\frac N2-\frac{4-b}\alpha$.  If $s_c=s$ (equivalently $\alpha=\frac{8-2b}{N-2s}$),  the Cauchy problem (\ref{NLS}) is known as $H^s$-critical; if $s_c>s$  (equivalently $0<\alpha$, $(N-2s)\alpha <8-2b$),  it is called $H^s$-subcritical.
The limiting case $b=0$ (classical biharmonic nonlinear Schr\"{o}dinger equation, also called the fourth-order Schr\"{o}dinger equation) has been introduced by Karpman \cite{Karpman1}, and Karpman--Shagalov \cite{Karpman2} to take into
account the role of small fourth-order dispersion terms in the propagation of intense laser beams in a bulk medium with Kerr
nonlinearity.  Since then, the study of nonlinear fourth-order Schr\"odinger equation
has been attracted a lot of interest in the past decade. See  \cite{Guo2,Gu,xuan,Miao,Miao2,Pa,Pa2} and  references cited therein.

The equation in (\ref{NLS}) has a counterpart for the Laplacian operator, namely, the inhomogeneous nonlinear Schr\"odinger equation
\begin{equation}\label{INLS}
	i\partial_tu+\Delta u+K(x)|u|^\alpha u=0.
\end{equation}
In Gill \cite{Gill} and Liu-Tripathi \cite{Liu}, it was suggested that stable high power propagation can be achieved in a plasma by sending a preliminary laser beam that creates a channel with a reduced electron density, and thus reduces the nonlinearity inside the channel. In
this case, the beam propagation can be modeled by the inhomogeneous nonlinear Schr\"odinger equation (\ref{INLS}). In addition, Fibich and Wang \cite{Fibich2} investigated (\ref{INLS}) for $K(\ep|x|)$ with $\ep$ small and $K\in C^4(\R^N)\cap L^\infty (\R^N)$, where  the solution $u$ is the electric field in laser and optics, and $K(x)$ is proportional to the electron density with a small parameter $\ep>0$ (see also \cite{Liu2}). For other interesting Physical applications of (\ref{INLS}), we refer to  \cite{Berge,Bo1,Bo2,Kar,Merle,Raphael,Towers} and the references therein.

Let us first review some known well-posedness results for (\ref{INLS}). We shall assume  $K(x)=\lambda \left|x\right|^{-b}$ with $\lambda \in\R$,  $0<b<2$ to make the review shorter.   Genoud and Stuart  \cite{Ge} first showed the local well-posedness in $H^{1}\left(\mathbb{R}^{N}\right)$ for $0<b<\min \{2, N\}$ and $0<\alpha$, $(N-2)\alpha<4-2b$  by using the abstract argument of Cazenave \cite{Ca9}, which does not use Strichartz's estimates.  In this case, Genoud \cite{Ge0} and Farah \cite{Fa} also showed
how small should be the initial data to have global well-posedness, respectively, in the spirit of Weinstein \cite{Wei} and Holmer-Roudenko \cite{Holmer} for the classical case $b=0$. Recently, using Strichartz's estimate  and the contraction mapping argument,  Guzm\'an  \cite{Gu2} showed the local well-posedness of  (\ref{INLS}) for $0<\alpha $, $(N-2)\alpha <4-2b$, but  under the  restrictions: $b<\frac{N}{3}$ if $N\le3$.  This restriction is a bit improved by Dinh \cite{Dinh} in  dimension $N=3$, for $0<b<\frac{3}{2}$ but for more restricted values $\alpha <\frac{6-4b}{2b-1}$. Although these results are a bit weak on the range of $b$ compared with the result of Genoud-Stuart, they provide more information on the solution due to the Strichartz's estimates. In particular, one can know    that the local solutions belong to $L^p_{loc}([0,T_{\text{max}}),L^q(\R^N))$ for any Schr\"odinger admissible pair $(p,q)$. In general, such property   plays an important role in studying other interesting problems, for instance, scattering and blow up. Finally, we review  the well-posedness of (\ref{INLS}) in $H^s$. Defining
\begin{equation}
	\hat 2:=\begin{cases}
		\min \left\{2,1+\frac{n-2s}{2}\right\}, \qquad &n\ge3,\\
		n-s,\qquad &n=1,2,
	\end{cases}
\end{equation}
it was proved in \cite{An} that, for $N\ge1$, $0<b<\hat 2$, and $0\le s<\frac{N}{2}$, $0<\alpha <\frac{4-2b}{N-2s}$ or $\frac{N}{2}\le s<\min \left\{N,\frac{N}{2}+1\right\} $, $0<\alpha <\infty $,  the Cauchy problem (\ref{INLS}) is local well-posed in $H^s(\R^N)$. Moreover, it was proved in \cite{Kim} that  (\ref{INLS}) is local well-posed in a weighted Sobolev space for  $N\ge3$, $0<s<\frac{1}{3}$, $\max \left\{\frac{26-3N}{12},\ \frac{12s+4Ns-8s^2}{N+4s}\right\}<b<2 $ and $\max \left\{0,\ \frac{10s-2\alpha }{N-6s}\right\}<\alpha \le \frac{4-2b}{N-2s} $.

In this paper we are interested in studying the well-posedness for (\ref{NLS}) in $H^s(\R^N)$, with $0<s\le2$.  This  problem  was firstly studied by C.M. Guzm\'an and A. Pastor \cite{Gu} for  $K(x)=\lambda \left|x\right|^{-b}$, $f(u)= \left|u\right|^\alpha u$, $\lambda\in\R$. They proved the local-well posedness in $H^2$ for $N\ge3$, $0<b<\min \left\{\frac{N}{2},4\right\}, \ 	\max \left\{0,\frac{2(1-b)}{N}\right\}<\alpha$, $ (N-4)\alpha <8-2b  $.  Also, they proved global well-posedness in the mass-subcritical and mass-critical cases in $H^2$, that is, $\min \left\{\frac{2(1-b)}{N},0\right\}<\alpha \le \frac{4-b}{N} $.  Afterwards, Cardoso--Guzman--Pastor \cite{Cardoso} established the local well-posedness in $\dot H^{s}\cap \dot H^2$ with $N\ge5$, $0<s<2$, $0<b<\min \left\{\frac{N}{2},4\right\} $ and $\max \left\{\frac{8-2b}{N},1\right\}<\alpha <\frac{8-2b}{N-4} $. Note that \cite{Gu,Cardoso} does not treat the low dimensions cases and there is  a lower bound for the parameter $\alpha$. The restrictions of dimensions and index $\alpha $ are due to the singularity of inhomogeneity coefficient and the lack of differentiability of the nonlinear term.  To resolve this, we derive  the following  bilinear Strichartz's type estimates for   nonlinear biharmonic Schr\"odinger equations in  Besov spaces,
by which we can replace the spatial derivative of order $s$ with the fractional
order time derivative of order $s/4$.  For the definition of the biharmonic admissible  set $\Lambda_b$, we refer to Section \ref{s2}.
\begin{proposition}\label{p1}
	Let  $N\ge1$, $\mu=0$ or $-1$,    $0<s\le2$, $(q,r)$, $(\gamma_0,\rho_0)$, $(\gamma,\rho)$, $(\gamma_1,\rho_1)\in\Lambda_b$ be four biharmonic admissible pairs, and $1\le \overline{q},\overline{r}\le2$ with  $\fc4{\overline q}-N(\fc12-\fc1{\overline r})=4-s$. Then for any $\varphi\in H^s(\R^N)$ and $f\in B^{s/4}_{\gamma_0',2}L^{\rho_0'}\cap L^{\overline q}L^{\overline r}$, we have $e^{it(\Delta^2+\mu\Delta)}\varphi, Gf\in C(\R,H^{s})$, where
	\begin{equation}
		(Gf)(t):=\int_0^te^{i(t-s)(\Delta^2+\mu\Delta)}f(s)ds.\notag
	\end{equation}
	Moreover, the following inequalities hold:
	\begin{equation}\label{311}
		\|e^{it(\Delta^2+\mu\Delta)}\varphi  \|_{L^qB^{s}_{r,2}\cap B^{s/4}_{q,2}L^r}\lesssim \left\|\varphi\right\|_{H^s},
	\end{equation}
	and
	\begin{eqnarray}\label{11214}
		\|G(fg)\|_{L^qB^{s}_{r,2}\cap B^{s/4}_{q,2}L^r}
		&\lesssim& \left(\int_{-\infty}^\infty\left(|\tau|^{-s/4}\|g(t-\tau)(f(t-\tau)-f(t))\|_{L^{\gamma_0'}L^{\rho_0'}}\right)^2\frac{d\tau}{|\tau|}\right)^{1/2}\notag\\
		&&+\left(\int_{-\infty}^\infty\left(|\tau|^{-s/4}\|(g(t-\tau)-g(t))f(t)\|_{L^{\gamma_1'}L^{\rho_1'}}\right)^2\frac{d\tau}{|\tau|}\right)^{1/2}\notag\\
		&&+\|fg\|_{L^{\gamma'}L^{\rho'}\cap L^{\overline q}L^{\overline r}}.
	\end{eqnarray}
\end{proposition}
\begin{remark}
	Proposition \ref{p1} is a refinement of the following inequality previously obtained by  Nakamura and Wada for classical Schr\"odinger equation in \cite{Na}:
	\begin{equation}
		\left\|Gf\right\|_{L^qB^s_{r,2}\cap B^{s/2}_{q,2}L^r}\lesssim \left\|f\right\|_{B^{s/2}_{\gamma',2}L^{\rho'}}+\left\|f\right\|_{L^{\overline{q}}L^{\overline{r}}},\notag
	\end{equation}
	where $0<s<2$, $(q,r)$, $(\gamma,\rho)$ are two classical Schr\"odinger admissible pairs and $1\le \overline{q},\overline{r}\le2$ with $\frac{2}{\overline{q}}-N\left(\frac{1}{2}-\frac{1}{\overline{r}}\right)=2-s$.
	The most important advantage of our new estimates (\ref{11214}) is that  we can choose the admissible
	pairs $(\gamma_0,\rho_0)$, $(\gamma_1,\rho_1)$ independently when the difference falls on different functions. Thus we can choose the index more flexible when establishing the nonlinear estimates needed for the contraction argument. This can be an advantage when we consider not only pure power but also more complicated nonlinear terms.
\end{remark}
Using the bilinear Strichartz's type estimate, the detailed technique of choosing indices (see Lemmas \ref{l1}--\ref{l10}) and applying various embeddings in nonlinear estimates, we can   establish the $H^s$ local well-posedness for  (\ref{NLS}) in the whole $H^s$ subcritical case, with $0<s\le2$.  Before stating our results, we give the following notation.
\begin{definition}
	Let $\alpha >0$,  $f\in C^1(\mathbb{C},\mathbb{C})$ in the real sense. We say that $f$ belongs to the class $\mathcal{ \alpha }$ if it satisfies $f(0)=0$ and
	\begin{equation}
		\left|f'(z_1)-f'(z_2)\right|\lesssim (|z_1|^{\alpha-1}+|z_2|^{\alpha-1})\left|z_1-z_2\right|,\qquad \forall z_1,z_2\in\mathbb{C}.
	\end{equation}
\end{definition}
\begin{remark}
	We note that the power type nonlinearities $f(u)=\lambda \left|u\right|^\alpha u$ and $f(u)=\lambda \left|u\right|^{\alpha +1}$ with $\lambda \in\mathbb{C}$, $\alpha >0$ are in the class $\mathcal{C}(\alpha )$. Moreover, for any $\alpha>0$ and  $f\in \mathcal{C}(\alpha)$, it is easy to check that the following inequality holds  for any $u,v\in\mathbb{C}$, 
	\begin{equation}\label{fu}
		\left|f(u)-f(v)\right|\lesssim \left(\left|u\right|^{\alpha}+\left|v\right|^{\alpha}\right)\left|u-v\right|.
	\end{equation}
\end{remark}
Now we are ready to state our main result.
\begin{theorem}\label{T1}
	Let $N\ge1$, $\mu=0$ or $-1$,  $\beta>\max \left\{2,\frac{N}{4}\right\} $, $0<s\le2$,   $0<\alpha$,  $\left(N-2s\right)\alpha <8-\frac{2N}{\beta}$, $f\in\mathcal{C}(\alpha )$, and $K(x)\in L^\infty (\R^N)+ L^\beta(\R^N)$.  Given $\varphi\in H^s(\R^N)$, there exists  $T_{\text{max}}(\varphi)\in(0,\infty]$ and a unique  maximum solution $ u\in C([0,T_{\text{max}}(\varphi)),H^s)\bigcap_{(q,r)\in\Lambda_b} L^{q}$ $\left.\left([0,T_{\text{max}}(\varphi)), B_{r, 2}^{s}\right) \right)$ to the  Cauchy problem
	(\ref{NLS}),   with  the following blowup alternative holds:
	\begin{equation}\label{bl1-1}
		\lim_{t\uparrow T_{\text{max}}}\|u(t)\|_{H^s}=\infty,\qquad \text{if } T_{\text{max}}(\varphi)<\infty.
	\end{equation}
	Moreover, if $\varphi_n\rightarrow\varphi$ in $H^s(\R^N)$ and  $u_n$ denotes the solution of (\ref{NLS}) with the initial value $\varphi_n$, then $u_n\rightarrow u$ in $C([0,A],H^s(\R^N))$ for  any  $0<A<T_{\text{max}}(\varphi).$
\end{theorem}

The argument used to derive Theorem \ref{T1} can also be applied to the classical inhomogeneous nonlinear Schr\"odinger equation (\ref{INLS}). More precisely, we can establish the  bilinear Strichartz's type estimate for (\ref{INLS}) in the spirit of Proposition \ref{p1}. Then we choose the index as in Lemmas \ref{l1}--\ref{l10} to  establish a series of estimates needed in the contraction argument.  Therefore this improves the above mentioned results in  \cite{An,Kim} on the validity of $\alpha $ and $b$ in the case $0<s<1$.

If $K(x)=\lambda \left|x\right|^{-b}$ and $f(u)= \left|u\right|^\alpha u$, with $ \lambda \in \mathbb{C},\ b>0,\ \alpha >0$, we then have the following result, which removes the lower bound $\alpha >\frac{2(1-b)}{N}$ made in \cite{Gu}.
\begin{corollary}\label{Co1}
	Let $N\ge1$, $0<b<\min\left\{\frac{N}{2},4\right\}$, $0<\alpha$, $\left(N-4\right)\alpha <8-2b$, $\lambda \in\mathbb{C}$ and $\mu=-1$ or $0$. Given $\varphi\in H^2(\R^N)$, there exists  $T_{\text{max}}(\varphi)\in(0,\infty]$ and a unique  maximum solution $ u\in C([0,T_{\text{max}}(\varphi)),H^2)\bigcap_{(q,r)\in\Lambda_b} L^{q}$ $\left.\left([0,T_{\text{max}}(\varphi)), B_{r, 2}^{2}\right) \right)$ to the  Cauchy problem
	\begin{equation}
		\begin{cases}
			i\partial_t  u+(\Delta^2+\mu\Delta)u+\lambda \left|x\right|^{-b}\left|u\right|^\alpha u=0,\\
			u(0,x)=\varphi(x),
		\end{cases}\notag
	\end{equation}
	with  the following blowup alternative holds:
	\begin{equation}
		\lim_{t\uparrow T_{\text{max}}}\|u(t)\|_{H^2}=\infty,\qquad \text{if } T_{\text{max}}(\varphi)<\infty.\notag
	\end{equation}
	Moreover, the continuous dependence upon the initial data holds.
\end{corollary}
If $K(x)=\lambda $ with $\lambda \in \mathbb{C}$ and $f(u)=\left|u\right|^\alpha u$,  we then have the following result for fourth-order Schr\"odinger equation.
\begin{corollary}\label{Co2}
	Let $N\ge1$, $0<s\le2$, $0<\alpha$, $\left(N-2s\right)\alpha <8$, $\lambda \in\mathbb{C}$ and $\mu=-1$ or $0$. Given $\varphi\in H^s(\R^N)$, there exists  $T_{\text{max}}(\varphi)\in(0,\infty]$ and a unique  maximum solution $ u\in C([0,T_{\text{max}}(\varphi)),H^s)\bigcap_{(q,r)\in\Lambda_b} L^{q}$ $\left.\left([0,T_{\text{max}}(\varphi)), B_{r, 2}^{s}\right) \right)$ to the  Cauchy problem
	\begin{equation}
		\begin{cases}
			i\partial_t  u+(\Delta^2+\mu\Delta)u+\lambda \left|u\right|^\alpha u=0,\\
			u(0,x)=\varphi(x),
		\end{cases}\notag
	\end{equation}
	with  the following blowup alternative holds:
	\begin{equation}
		\lim_{t\uparrow T_{\text{max}}}\|u(t)\|_{H^s}=\infty,\qquad \text{if } T_{\text{max}}(\varphi)<\infty.\notag
	\end{equation}
	Moreover, the continuous dependence upon the initial data holds.
\end{corollary}
\begin{remark}
	Corollary \ref{Co2} improves the corresponding results of \cite{Guo2,Dinh2} in the case $0<s\le2$. In \cite{Guo2},  there is a additional assumption $s\le [\alpha ]$ ($[\alpha ]$ denotes the largest integer  less than or equal to $\alpha $) for the parameter $\alpha $; and in \cite{Dinh}, the continuous dependence   ($u_n\rightarrow u$ in $C([0,T],H^{s-\ep})$ with $\ep>0$)  is weaker than the expected one.
\end{remark}
In the sequel, we establish global existence results in energy space in the $L^2$-subcritical regime. We assume that $K(x)$ is a real-valued function, which will be used to establish the conservation of the mass.
\begin{theorem}\label{T2}
	Let $N\ge1$,  $\mu=0$ or $-1$,  $\beta>\max \left\{2,\frac{N}{4}\right\} $,  $0<\alpha \le\frac{8}{N}-\frac{2}{\beta}$,   $K(x)\in L^\infty (\R^N)+L^\beta(\R^N)$ be a real-valued function, and $f\in\mathcal{C}(\alpha )$ that satisfies  	\\
	(i) $f(a)\in \R$ for all $a\ge0$; \\
	(ii) $f(u)=\frac{u}{\left|u\right|}f(\left|u\right|)$ for all $u\in\mathbb{C}\setminus \left\{0\right\} $;\\
	Then the local solution   obtained in Theorem  \ref{T1} with the initial datum $\varphi$ can be extended  globally-in-time if one of the
	following alternatives holds:\\
	(i) $0<\alpha <\frac{8}{N}-\frac{2}{\beta}$, or\\
	(ii) $\alpha =\frac{8}{N}-\frac{2}{\beta}$ and $\left\|\varphi\right\|_{L^2}$ sufficiently small.
\end{theorem}
If $K(x)=\lambda \left|x\right|^{-b}$ and $f(u)= \left|u\right|^\alpha u$ with  $\lambda \in \mathbb{R}$, we then have the following result, which removes the lower bound $\alpha >\frac{2(1-b)}{N}$ made in \cite{Gu}.
\begin{corollary}
	Let $N\ge1$, $0<b<\min\left\{\frac{N}{2},4\right\}$, $0<\alpha\le\frac{8-2b}{N}$,  and  $ K(x)=\lambda |x|^{-b}$, $f(u)=|u|^\alpha u$, $\lambda\in\R$.  Then the local solution   obtained in Corollary  \ref{Co1} with the initial datum $\varphi$ can be extended  globally-in-time if one of the
	following alternatives holds:\\
	(i) $0<\alpha <\frac{8-2b}{N}$, or\\
	(ii) $\alpha =\frac{8-2b}{N}$ and $\left\|\varphi\right\|_{L^2}$ sufficiently small.
\end{corollary}
When $K(x)$ is a complex-valued function, the solutions to (\ref{NLS}) may blow up in finite time. In fact, for any given compact set $M\subset\R^N$, it was proved in \cite{xuan3} that, for $K(x)=\lambda, \ f(u)=\left|u\right| ^\alpha u$ with $\text{Im}\lambda >0,\ 0<\alpha <\frac{8}{N-8}$, there exists a class of  solutions to (\ref{NLS}), which blows up exactly on $M$.

The rest of the paper is organized as follows. In Section \ref{s2}, we introduce some notations and give a review of the biharmonic Strichartz's estimates. In Section \ref{s3}, we establish the  bilinear   Strichartz's type estimate. In Section \ref{s4}, we establish the nonlinear estimates that are needed in the contraction argument. In Section \ref{5}, we give the proof of Theorem \ref{T1}
and Theorem \ref{T2}.
\section{Preliminary}\label{s2}
If $X, Y$ are nonnegative quantities, we sometimes use $X\lesssim Y$ to denote the estimate $X\leq CY$ for some positive constant $C$. Pairs of conjugate indices are written as $p$ and $p'$, where $1\leq p\leq\infty$ and $\frac1p+\frac1{p'}=1$. We use $L^p (\mathbb{R}^N)$ to denote the usual Lebesgue space and  $L^\gamma(I,L^\rho(\mathbb{R}^N))$ to denote the space-time Lebesgue spaces with the  norm
\begin{gather}\notag
	\|f\|_{L^\gamma(I,L^\rho(\mathbb{R}^N))}:=\left(\int_I\|f\|_{L_x^\rho}^\gamma dt\right)^{1/\gamma}
\end{gather}
for any time slab $I\subset\mathbb{R}$, with the usual modification when either $\gamma$ or $\rho$ is infinity.
We also define the  Fourier transform on $\R,\R^N$ and $\R^{1+N}$   by
\begin{align*}
	&\hat f(\tau)=\int_\R f(t)e^{-it\tau}dt,\qquad\qquad\quad\qquad\qquad\tau\in\R;\\
	&\hat f(\xi)=\int_{\R^N}f(x)e^{-ix\cdot\xi}d x,\qquad\quad\qquad\qquad\xi\in\R^N;\\
	&\widetilde f(\tau,\xi)=\int_{\R^{1+N}}f(t,x) e^{-ix\cdot\xi-it\tau}dxdt,\quad(\tau,\xi)\in\R\times\R^N,
\end{align*}
respectively.

Next, we  review the definition of Besov spaces.  Let $\phi$ be a smooth function whose Fourier transform $\hat\phi$ is a non-negative even function which satisfies supp $\hat\phi\subset\{\tau\in\R,1/2\le|\tau|\le2\}$ and $\sum_{k=-\infty}^{\infty}\hat\phi(\tau/{2^k})=1$ for any $\tau\neq0$. For $k\in\mathbb{Z}$, we put  $\hat\phi_k(\cdot)=\hat\phi(\cdot/{2^k})$ and $\psi=\sum_{j=-\infty}^0\phi_j$.   Moreover, we define  $\chi_k=\sum_{k-2}^{k+2}\phi_j$ for $k\ge1$ and $\chi_0=\psi+\phi_1+\phi_2$.  For $s\in\R$ and $1 \leq p$, $q \leq \infty$, we define  the Besov space
$$
B_{p,q}^{s}\left({\R}^{N}\right)=\left\{u \in \mathcal{S}^{\prime}\left(\R^{N}\right), \|u\|_{B_{p,q}^{s}\left(\R^{N}\right)}<\infty\right\}
$$
where $\mathcal{S}^{\prime}\left({\R}^{N}\right)$ is the space of tempered distributions on $\R^{N},$ and
$$
\|u\|_{B_{p, q}^{s}\left(\R^{N}\right)}=\left\|\psi *_{x} u\right\|_{L^{p}\left(\R^{N}\right)}+\left\{\begin{array}{ll}
	\left\{\sum_{k \geq 1}\left(2^{s k}\left\|\phi_{k} *_{x} u\right\|_{L^{p}\left(\R^{N}\right)}\right)^{q}\right\}^{1 / q}, & q<\infty, \\
	\sup _{k \geq 1} 2^{s k}\left\|\phi_{k}*_{x} u\right\|_{L^{p}\left(\R^{N}\right)}, & q=\infty,
\end{array}\right.
$$
where $*_{x}$ denotes the convolution with respect to the variables in $\R^{N}$.  Here we use $\phi_k*_xu$ to denote $\phi_k(|\cdot|)*_xu$. We also define $\chi_k*_xu,\psi*_xu,\chi_0*_xu$ similarly. This is an abuse of symbol, but no confusion is likely to arise.

For $1\le q$, $\alpha\le \infty $ and a Banach space $V$, we denote the  vector-valued  Lebesgue space for functions on $\R$ to $V$ by $L^q\left(\R,V\right)$. Then we define the  vector-valued  Sobolev space $$H^{1,q}\left(\R,V\right)=\left\{u:u\in L^q\left(\R,V\right),\partial_tu\in L^q \left(\R,V\right)\right\}.$$
Finally, we  define the Besov space of vector-valued functions. Let $\theta \in\R$, $1 \leq q$, $\alpha \leq \infty$ and $V$ be a Banach space. We put
$$
B_{q, \alpha}^{\theta}(\R, V)=\left\{u \in \mathcal{S}^{\prime}(\R, V) ;\|u\|_{B_{q, \alpha}^{\theta}(\R, V)}<\infty\right\},
$$
where
\begin{equation}\label{9221}
	\|u\|_{B_{q, \alpha}^{\theta}(\R, V)}=\left\|\psi *_{t} u\right\|_{L^{q}(\R, V)}+\left\{\sum_{k \geq 1}\left(2^{\theta k}\left\|\phi_{k} *_{t} u\right\|_{L^{q}(\R, V)}\right)^{\alpha}\right\}^{1 / \alpha}
\end{equation}
with trivial modification if $\alpha=\infty.$ Here $*_{t}$ denotes the convolution in
$\R$. Moreover, it is well-known that the norm (\ref{9221}) has the following equivalence
\begin{equation}
	\|u\|_{B_{q, \alpha}^{\theta}(\R, V)}\approx \left\|u\right\|_{L^q(\R,V)}+\left(\int_{-\infty }^{\infty }\left(\tau^{-s}\left\|u(t)-u(t-\tau)\right\|_{L^q(\R,V)}\right)^q \frac{d \tau}{\left|\tau\right|}  \right)^{1/q}. \notag
\end{equation}

Following standard notations, we introduce Schr\"odinger admissible pair as well as the corresponding Strichartz's estimate for the biharmonic Schr\"odinger equation.
\begin{definition}\label{bpair}
	A pair of Lebesgue space exponents $(\gamma, \rho)$ is called  biharmonic Schr\"odinger admissible for the equation (\ref{NLS}) if  $(\gamma, \rho)\in \Lambda_b$ where
	\begin{equation*}
		\Lambda_b=\{(\gamma, \rho):2\leq \gamma, \rho\leq\infty, \  \frac4\gamma+\frac N\rho=\frac N2, \  (\gamma, \rho, N)\neq(2, \infty, 4)\}.
	\end{equation*}
\end{definition}
\begin{lemma}[Strichartz's estimate  for BNLS]\label{L2.2S}
	Suppose that $(\gamma,\rho)$, $(a,b)\in\Lambda_b $ are  two biharmonic admissible pairs, and  $\mu=0$ or $-1$. Then for any  $u\in L^2(\mathbb{R}^N)$ and $h\in L^{a'}(\R,L^{b'}(\mathbb{R}^N))$, we have
	\begin{gather}\label{sz}
		\|e^{it(\Delta^2+\mu\Delta)}u\|_{L^\gamma L^\rho}\lesssim \|u\|_{L^2},
	\end{gather}
	\begin{equation}\label{SZ}
		\left\|\int_0^te^{i(t-s)(\Delta^2+\mu\Delta)}h(s)\ ds\right\|_{L^\gamma L^\rho}\lesssim  \|h\|_{L^{a'}L^{b'}},
	\end{equation}
	\begin{equation}
		\label{SZ1}
		\left\|\int_t^\infty e^{i(t-s)(\Delta^2+\mu\Delta)}h(s)\ ds \right\|_{L^\gamma L^\rho}\lesssim \|h\|_{L^{a'}L^{b'}}.
	\end{equation}
\end{lemma}
\begin{proof}
	The estimates (\ref{sz}) and (\ref{SZ1}) are proved in \cite{Pa}; and the proof of (\ref{SZ1}) is follows from an obvious adaptation of Corollary 2.3.6 in \cite{Ca9}.
\end{proof}
In this paper, we  omit the integral domain for simplicity unless noted otherwise. For example, we write $ L^qL^r= L^q\left(\R,L^r(\R^N)\right)$,  $L^qB^s_{r,2}=L^q\left(\R,B^s_{r,2}(\R^N)\right)$ and $B^\theta_{q,2}L^r=B^\theta_{q,2}(\R,L^r(\R^N))$ etc.

\section{Bilinear  Strichartz's type  estimate }\label{s3}
In this section we  prove Proposition \ref{p1}. First, we prepare several lemmas. We assume the functions $\phi,\chi_0,\psi,\phi_j,\chi_j$ are defined in Section \ref{s2}.
\begin{lemma}[Lemma 3.1 in \cite{xuan}]\label{L1}
	Assume $N\ge1$, $\mu=-1$ or $0$, and  $K_j(t,x)(j\ge1):\R\times\R^N\rightarrow\mathbb{C}$ are defined by
	\begin{equation}
		K_j(t,x)=\fc1{(2\pi)^{1+N}}\int e^{it\tau+ix\cdot \xi}\fc{\hat\phi_j(|\xi|^4-\mu|\xi|^2)(1-\hat\chi_j(\tau))}{i(\tau-|\xi|^4+\mu|\xi|^2)} d\tau  d\xi.\notag
	\end{equation}
	Then for any $0<s<4$, $1\le q\le\infty$, $1\le r\le\infty$ with $\fc4{q}-N(1-\fc1{ r})=s$, we have
	\begin{equation*}
		\|K_j\|_{L^{ q}L^{ r}}\le C2^{-js/4},
	\end{equation*}
	where the constant $C$ is independent of $j\ge1$.
\end{lemma}
\begin{lemma}[Lemma 3.3 in \cite{xuan}]\label{dj}
	Let $s\in\R$, $1\le p$, $q\le\infty$, then the norm defined by
	\begin{eqnarray*}
		&&\|u\|_{\widetilde{B }_{p, q}^{s}\left(\mathbb{R}^{n}\right)}:=\left\|\left(\mathcal{F}_{\xi}^{-1}\left(\widehat{\psi}\left(|\xi|^{4}-\mu|\xi|^2\right)\right)\right) *_{x} u\right\|_{L^{p}\left(\mathbb{R}^{n}\right)} \\
		&&+\begin{cases}\left\{\sum_{j \geq 1}\left(2^{s j / 4}\left\|\left(\mathcal{F}_{\xi}^{-1}\left(\widehat{\phi}\left((|\xi|^{4}-\mu|\xi|^2)/{2^{j}}\right)\right)\right) *_{x} u\right\|_{L^{p}\left(\mathbb{R}^{n}\right)}\right)^{q}\right\}^{1 / q},\text{ if }q<\infty,\\
			\sup _{j \geq 1} 2^{s j / 4}\left\|\left(\mathcal{F}_{\xi}^{-1}\left(\widehat{\phi}\left((|\xi|^{4}-\mu|\xi|^2)/{2^{j}}\right)\right)\right) *_{x} u\right\|_{L^{p}\left(\mathbb{R}^{n}\right)},\text{ if }q=\infty,\end{cases}
	\end{eqnarray*}
	is equivalent to the norm $\|u\|_{B^s_{p,q}}(\R^N)$ for any function $u$.
\end{lemma}
In the rest of this section, we use the natation $\phi_{ j/4}=\mathcal{F}_\xi^{-1}\left(\hat\phi_j(|\xi|^4-\mu|\xi|^2)\right)$. This is an abuse of symbol, but no confusion is likely to arise. Under this notation, we obtain the following equivalence from Lemma \ref{dj}
\begin{equation}\label{1271}
	\|u\|_{{B}_{p, q}^{s}}\approx\left\|\left(\mathcal{F}_{\xi}^{-1}\left(\widehat{\psi}\left(|\xi|^{4}-\mu|\xi|^2\right)\right)\right) *_{x} u\right\|_{L^{p}}
	+\left(\sum_{j=1}^\infty\left(2^{s j / 4}\left\| \phi_{ j/4}*_{x} u\right\|_{L^{p}}\right)^{q}\right)^{1 / q}
\end{equation}
with trivial modification if $q=\infty$.
\begin{lemma}[Lemma 2.1 in \cite{Pe}]\label{31}
	Suppose that $N\ge1$,  $1\leq\gamma<\infty$, $1\leq\rho\leq\infty$, $0<\theta<1,$ and $1\leq q\leq\infty$, then we have
	\begin{equation}\left(L^{\gamma}\left(\R, L^{\rho}(\R^N)\right), H^{1, \gamma}\left(\R, L^{\rho}(\R^N)\right)\right)_{\theta, q}=B_{\gamma, q}^{\theta}\left(\R, L^{\rho}(\R^N)\right).\notag
	\end{equation}
\end{lemma}
\begin{proof}[\textbf{Proof of Proposition \ref{p1}}]
	The continuity of $Gf$ and $e^{it(\Delta^2+\mu\Delta)}\varphi $ in time follows from density argument. We now prove the inequality (\ref{311}).
	Using the similar argument as that used in the proof of Corollary 2.3.9 in \cite{Ca9}, we obtain the following estimates and omit the details:
	\begin{equation}
		\left\|e^{i t (\Delta^2+\mu\Delta)} \phi\right\|_{L^{q}\left(\R, B_{r, 2}^{\widetilde{s}}\right)} \lesssim\|\phi\|_{H^{\widetilde{s}}},\label{321}
	\end{equation}
	\[
	\left\|\frac{d}{d t}\left(e^{i t (\Delta^2+\mu\Delta)} \phi\right)\right\|_{L^{q}\left(\R, B_{r, 2}^{\widetilde{s}}\right)} \lesssim \|\phi\|_{H^{\widetilde{s}+4}},
	\]
	for any $\widetilde{s}>0$ and $(q,r)\in \Lambda _b$. Define the operator  $\mathcal{G}f:=e^{it(\Delta^2+\mu\Delta)}f$, then we have
	\begin{eqnarray}
		&&\mathcal{G}:L^2\rightarrow  L^{q}\left(\R, L^{r}\right),\notag\\
		&&\mathcal{G}:H^4\rightarrow  H^{1,q}\left(\R, L^{r}\right). \notag
	\end{eqnarray}
	Moreover, from the  interpolation theorem and Lemma \ref{31}, we have
	\begin{equation}
		\mathcal{G}:\left(L^2,H^4\right)_{s/4,2}\rightarrow \left(L^{q}\left(\R, L^{r}\right)\cap H^{1, q}\left(\R, L^{r}\right)\right)_{s/4,2}= B^{s/4}_{q,2}\left(\R,L^r\right).\label{322}
	\end{equation}
	The inequality (\ref{311}) is now an immediate consequence of (\ref{321}), (\ref{322}) and $\left(L^2,H^4\right)_{s/4,2}$ $=H^{s}$.

	In what follows, we prove the estimate (\ref{11214}).
	Taking the Fourier transform, we get
	\begin{equation}\label{9206}
		(G(fg))\hat{\phantom{fg}}(t,\xi)=\int_{-\infty}^{\infty}\fc{e^{it\tau}-e^{it(|\xi|^4-\mu|\xi|^2)}}{2\pi i(\tau-|\xi|^4+\mu|\xi|^2)}\widetilde{fg}(\tau,\xi)d \tau.
	\end{equation}
	From (\ref{9206}) and
	$\phi_j*_te^{ita}=e^{ita}\hat\phi_j(a)$, $\forall$ $a\in\R,$ we obtain, for any $j\ge1$,
	\begin{eqnarray}
		&&\phi_j*_t(G(fg))\hat{\phantom{f}}(t,\xi)\notag\\
		&=&\int_{-\infty}^{\infty}\fc{e^{it\tau}\hat\phi_j(\tau)}{2\pi i(\tau-|\xi|^4+\mu|\xi|^2)}\widetilde {fg}(\tau,\xi) d\tau\notag\\
		&&-\int_{-\infty}^{\infty}\fc{e^{it(|\xi|^4-\mu|\xi|^2)}\hat\phi_j(|\xi|^4-\mu|\xi|^2)\hat\chi_j(\tau)}{{2\pi i(\tau-|\xi|^4+\mu|\xi|^2)}}\widetilde {fg}(\tau,\xi) d\tau\notag\\
		&&-\int_{-\infty}^{\infty}\fc{e^{it(|\xi|^4-\mu|\xi|^2)}\hat\phi_j(|\xi|^4-\mu|\xi|^2)(1-\hat\chi_j(\tau))}{2\pi i(\tau-|\xi|^4+\mu|\xi|^2)}\cdot \hat\chi_j(|\xi|^4-\mu|\xi|^2)\widetilde {fg}(\tau,\xi) d\tau,\notag
	\end{eqnarray}
	where we also used the fact that $\hat\chi_k=1$ on the support of $\hat\phi_k$.
	Moreover, since $$\mathcal{F}_\tau^{-1}\{\fc1{i(\tau-|\xi|^4+\mu|\xi|^2)}\}(t)=\fc12\text{sign}(t)e^{it(|\xi|^4-\mu|\xi|^2)},$$ we obtain
	\begin{eqnarray}\label{2410}
		&&\phi_j*_t(G(fg))\notag\\
		&=&\fc12\int_{-\infty}^{\infty}\text{sign}(t-\tau)e^{i(t-\tau)(\Delta^2+\mu\Delta)}(\phi_j*_t(fg))(\tau)d\tau\notag\\
		&&-\fc12e^{it(\Delta^2+\mu\Delta)}\int_{-\infty}^{\infty}\text{sign}(-\tau)e^{i\tau(\Delta^2+\mu\Delta)}
		(\phi_{ j/4}*_x\chi_j*_t(fg))(\tau) d\tau\notag \\
		&&-e^{it(\Delta^2+\mu\Delta)}\{K_j*_{t,x}\chi_{ j/4}*_x(fg)\}|_{t=0},
	\end{eqnarray}
	where $K_j$ is the function  defined in Lemma \ref{L1}.
	
	We first prove that
	\begin{eqnarray}\label{11212}
		&&\left(\sum_{j\ge1}\left(2^{js/4}\|K_j*_{t,x}\chi_{j/4}*_x(fg)\|_{L^\infty L^2\cap L^q L^r}\right)^2\right)^{1/2}\lesssim  \left\|fg\right\|_{L^{\overline{q}}L^{\overline{r}}}.
	\end{eqnarray}
	Let $q_0,r_0,q_1,r_1$ be given by the equation $1=\frac{1}{q_0}+\frac{1}{\overline{q}}$, $1+\frac{1}{2}=\frac{1}{r_0}+\frac{1}{\overline{r}}$, $1+\frac{1}{q}=\frac{1}{q_1}+\frac{1}{\overline{q}}$ and $1+\frac{1}{r}=\frac{1}{r_1}+\frac{1}{\overline{r}}$. Then it is easy to check that $1\le q_0,r_0,q_1,r_1\le \infty $ and $\frac{4}{q_0}-N\left(1-\frac{1}{r_0}\right)=\frac{4}{q_1}-N\left(1-\frac{1}{r_1}\right)=s$. Thus from Young's inequality and Lemma \ref{L1}, we have
	\begin{eqnarray}\label{2241}
		\left\|K_j*_{t,x}\chi_{j/4}*_x(fg)\right\|_{L^\infty L^2\cap L^qL^r}&\lesssim& \left\|K_j\right\|_{L^{q_0}L^{r_0}\cap L^{q_1}L^{r_1}}\left\|\chi_{j/4}*_x(fg)\right\|_{L^{\overline{q}}L^{\overline{r}}}\notag\\
		&\lesssim& 2^{-js/4}\left\|\chi_{j/4}*_x(fg)\right\|_{L^{\overline{q}}L^{\overline{r}}}.
	\end{eqnarray}
	Since $1\le \overline{q},\overline{r}\le2$, it follows from (\ref{2241}), Minkowski's inequality and Sobolev's embedding $L^{\overline{r}}(\R^N)\hookrightarrow B^0_{\overline{r},2}(\R^N)$ that
	\begin{eqnarray}
		&&\left(\sum_{j\ge1}\left(2^{js/4}\|K_j*_{t,x}\chi_{j/4}*_x(fg)\|_{L^q L^r}\right)^2\right)^{1/2}\notag\\
		&\lesssim &  \left(\sum_{j\ge1}\left\|\chi_{j/4}*_x(fg)\right\|_{L^{\overline{q}}L^{\overline{r}}}^2\right)^{1/2}\lesssim \left\|fg\right\|_{L^{\overline{q}}B^0_{\overline{r},2}}\lesssim \left\|fg\right\|_{L^{\overline{q}}L^{\overline{r}}}.
	\end{eqnarray}
	
	Next, we prove that
	\begin{eqnarray}\label{11210}
		&&\left(\sum_{j\ge1}\left(2^{js/4}\|\int_{-\infty}^{\infty}\text{sign}(t-\tau)e^{i(t-\tau)(\Delta^2+\mu\Delta)}(\phi_j*_t(fg))d\tau\|_{L^q L^r}\right)^2\right)^{1/2}\notag\\
		&\lesssim &\left(\int_{-\infty}^\infty\left(|\tau|^{-s/4}\|g(t-\tau)(f(t-\tau)-f(t))\|_{L^{\gamma_0'}L^{\rho_0'}}\right)^2\fc{d\tau}{|\tau|}\right)^{1/2}\notag\\
		&&+\left(\int_{-\infty}^\infty\left(|\tau|^{-s/4}\|(g(t-\tau)-g(t))f(t)\|_{L^{\gamma_1'}L^{\rho_1'}}\right)^2\fc{d\tau}{|\tau|}\right)^{1/2}.
	\end{eqnarray}
	Since $\int_{-\infty }^\infty  \phi(t) dt=\hat\phi(0)=0$, we have
	\begin{eqnarray}\label{248}
		\phi_j*_t(fg)(t)&=&2^j\int_{-\infty }^\infty \phi(2^j\tau)g(t-\tau)(f(t-\tau)-f(t))d\tau \notag\\
		&&+2^j\int_{-\infty }^\infty \phi(2^j\tau)(g(t-\tau)-g(t))f(t)d\tau.
	\end{eqnarray}
	From (\ref{248}) and Strichartz's estimate (\ref{SZ1}), we get
	\begin{eqnarray}\label{249}
		&&\sum_{j\ge1}\left(2^{js/4}\|\int_{-\infty}^{\infty}\text{sign}(t-\tau)e^{i(t-\tau)(\Delta^2+\mu\Delta)}(\phi_j*_t(fg))d\tau\|_{L^q L^r}\right)^2\notag\\
		&\lesssim&\sum_{j\ge1}2^{j(s/2+2)}\left(\int_{-\infty }^\infty |\phi(2^j\tau)|\|g(t-\tau)(f(t-\tau)-f(t))\|_{L^{\gamma_0'}L^{\rho_0'}}d\tau\right)^2\notag\\
		&&+\sum_{j\ge1}2^{j(s/2+2)}\left(\int_{-\infty }^\infty |\phi(2^j\tau)|\|(g(t-\tau)-g(t))f(t)\|_{L^{\gamma_1'}L^{\rho_1'}}d\tau\right)^2\notag\\
		&=:&\uppercase\expandafter{\romannumeral1}+\uppercase\expandafter{\romannumeral2}.
	\end{eqnarray}
	For the estimate of $\uppercase\expandafter{\romannumeral1}$, we have,  by using Cauchy-Schwartz inequality,
	\begin{eqnarray}
		&&\left(\int_{-\infty }^\infty |\phi(2^j\tau)|\|g(t-\tau)(f(t-\tau)-f(t))\|_{L^{\gamma_0'}L^{\rho_0'}}d\tau\right)^2\notag\\		&\lesssim&\int_{2^j|\tau|\le1}\left(\phi(2^j\tau)\right)^2d\tau\int_{2^j|\tau|\le1}\|g(t-\tau)(f(t-\tau)-f(t))\|_{L^{\gamma_0'}L^{\rho_0'}}^2d\tau\notag\\		&&+\int_{2^j|\tau|\ge1}|\tau|^3\left(\phi(2^j\tau)\right)^2d\tau\int_{2^j|\tau|\ge1}\fc{\|g(t-\tau)(f(t-\tau)-f(t))\|^2_{L^{\gamma_0'}L^{\rho_0'}}}{|\tau|^3}d\tau\notag\\
		&\lesssim&2^{-j}\int_{2^j|\tau|\le1}\|g(t-\tau)(f(t-\tau)-f(t))\|_{L^{\gamma_0'}L^{\rho_0'}}^2d\tau\notag\\
		&&+2^{-4j}\int_{2^j|\tau|\ge1}\fc{\|g(t-\tau)(f(t-\tau)-f(t))\|^2_{L^{\gamma_0'}L^{\rho_0'}}}{|\tau|^3}d\tau.\notag
	\end{eqnarray}
	This inequality together with Fubini's Theorem yields
	\begin{eqnarray}\label{1127}		\uppercase\expandafter{\romannumeral1}&\lesssim&\sum_{j\ge1}2^{j(s/2+1)}\int_{2^j|\tau|\le1}\|g(t-\tau)(f(t-\tau)-f(t))\|_{L^{\gamma_0'}L^{\rho_0'}}^2d\tau\notag\\
		&&+\sum_{j\ge1}2^{j(s/2-2)}\int_{2^j|\tau|\ge1}\fc{\|g(t-\tau)(f(t-\tau)-f(t))\|^2_{L^{\gamma_0'}L^{\rho_0'}}}{|\tau|^3}d\tau\notag\\
		&\lesssim&\int_{-\infty}^\infty\sum_{2^j|\tau|\le1}2^{j(s/2+1)}\|g(t-\tau)(f(t-\tau)-f(t))\|_{L^{\gamma_0'}L^{\rho_0'}}^2d\tau\notag\\
		&&+\int_{-\infty}^\infty\sum_{2^j|\tau|\ge1}2^{j(s/2-2)}\fc{\|g(t-\tau)(f(t-\tau)-f(t))\|^2_{L^{\gamma_0'}L^{\rho_0'}}}{|\tau|^3}d\tau\notag\\
		&\lesssim&\int_{-\infty}^\infty\left(|\tau|^{-s/4}\|g(t-\tau)(f(t-\tau)-f(t))\|_{L^{\gamma_0'}L^{\rho_0'}}\right)^2\fc{d\tau}{|\tau|}.
	\end{eqnarray}
	Similarly, we have
	\begin{equation}\label{1128}		\uppercase\expandafter{\romannumeral2}\lesssim\int_{-\infty}^\infty\left(|\tau|^{-s/4}\|(g(t-\tau)-g(t))f(t)\|_{L^{\gamma_1'}L^{\rho_1'}}\right)^2\fc{d\tau}{|\tau|}.
	\end{equation}
	The inequality (\ref{11210}) is now an immediate consequence of  (\ref{249}),  (\ref{1127}) and (\ref{1128}).
	
	Using the same method as that used to derive (\ref{11210}), we obtain
	\begin{align}\label{11211}		&\left(\sum_{j\ge1}(2^{js/4}\|e^{it(\Delta^2+\mu\Delta)}\int_{-\infty}^{\infty}\text{sign}(-\tau)e^{i\tau(\Delta^2+\mu\Delta)}(\phi_{j/4}*_x\chi_j*_t(fg))(\tau) d\tau\|_{L^q L^r})^2\right)^{1/2}\notag\\
		&\lesssim \left(\int_{-\infty}^\infty\left(|\tau|^{-s/4}\|g(t-\tau)(f(t-\tau)-f(t))\|_{L^{\gamma_0'}L^{\rho_0'}}\right)^2\fc{d\tau}{|\tau|}\right)^{1/2}\notag\\
		&\ \qquad+\left(\int_{-\infty}^\infty\left(|\tau|^{-s/4}\|(g(t-\tau)-g(t))f(t)\|_{L^{\gamma_1'}L^{\rho_1'}}\right)^2\fc{d\tau}{|\tau|}\right)^{1/2}.
	\end{align}

	Since $\left\|\psi*_tG(fg)\right\|_{L^qL^r}\lesssim \left\|fg\right\| _{L^{\gamma'}L^{\rho'}}$ by Young's inequality and Strichartz's estimate (\ref{SZ}), it follows from (\ref{2410}), (\ref{11212}), (\ref{11210}), (\ref{11211}) and Strichartz's estimate (\ref{sz}) that
	\begin{eqnarray}\label{1141}
		\|G(fg)\|_{B^{s/4}_{q,2}L^r}
		&\lesssim &\left\{\int_{-\infty}^\infty\left(|\tau|^{-s/4}\|g(t-\tau)(f(t-\tau)-f(t))\|_{L^{\gamma_0'}L^{\rho_0'}}\right)^2\fc{d\tau}{|\tau|}\right\}^{\fc12}\notag\\
		&&+\left\{\int_{-\infty}^\infty\left(|\tau|^{-s/4}\|(g(t-\tau)-g(t))f(t)\|_{L^{\gamma_1'}L^{\rho_1'}}\right)^2\fc{d\tau}{|\tau|}\right\}^{\fc12}\notag\\
		&&+\left\|fg\right\| _{L^{\gamma'}L^{\rho'}\cap L^{\overline{q}}L^{\overline{r}}}.
	\end{eqnarray}
	
	Finally, we estimate $\left\|G(fg)\right\|_{L^qB^s_{r,2}}$.
	Similar to (\ref{2410}), we can write
	\begin{eqnarray}\label{11213}
		\phi_{ j/4}*_x(G(fg))&=&\fc12\int_{-\infty}^{\infty}\text{sign}(t-\tau)e^{i(t-\tau)(\Delta^2+\mu\Delta)}(\phi_{ j/4}*_x\chi_j*_t(fg))(\tau)d\tau\notag\\
		&&+K_j*_{t,x}\chi_{ j/4}*_x(fg)\notag\\
		&&-\fc12e^{it(\Delta^2+\mu\Delta)}\int_{-\infty}^{\infty}\text{sign}(-\tau)e^{i\tau(\Delta^2+\mu\Delta)}(\phi_{ j/4}*_x\chi_j*_t(fg))(\tau) d\tau \notag\\
		&&-e^{it(\Delta^2+\mu\Delta)}\{K_j*_{t,x}\chi_{ j/4}*_x(fg)\}|_{t=0}.\notag
	\end{eqnarray}
	Since $\int_{- \infty}^{\infty }\chi(\tau)d\tau=\hat\chi(0)=0 $, we can apply the equivalent norm in   (\ref{1271}) and the same argument as that used to derive (\ref{1141})  to obtain
	\begin{eqnarray}
		\|G(fg)\|_{L^qB^{s}_{r,2}}
		&\lesssim &\left\{\int_{-\infty}^\infty\left(|\tau|^{-s/4}\|g(t-\tau)(f(t-\tau)-f(t))\|_{L^{\gamma_0'}L^{\rho_0'}}\right)^2\fc{d\tau}{|\tau|}\right\}^{\fc12}\notag\\
		&&+\left\{\int_{-\infty}^\infty\left(|\tau|^{-s/4}\|(g(t-\tau)-g(t))f(t)\|_{L^{\gamma_1'}L^{\rho_1'}}\right)^2\fc{d\tau}{|\tau|}\right\}^{\fc12}\notag\\
		&&+\left\|fg\right\| _{L^{\gamma'}L^{\rho'}\cap L^{\overline{q}}L^{\overline{r}}}.
	\end{eqnarray}
	This inequality together with (\ref{1141}) finishes the proof of  Proposition \ref{p1}.
\end{proof}
\section{Nonlinear estimates}\label{s4}
In this section, we prove the  following lemmas, which provides an estimate for the nonlinearity in the Strichartz spaces. Before stating the Lemmas, we define
$$
1_{\alpha<1}=\begin{cases}1,\quad\text{if }0<\alpha<1;\\
	0, \quad\text{if }\alpha\ge1,\end{cases}\qquad 4^*=\begin{cases}
	\frac{2N}{N-4},\quad \text{if }N\ge5,\\
	\infty,\quad \text{if }1\le N\le 4.
\end{cases}
$$
We also define the norm,
\begin{eqnarray}
	\left\|u\right\|_{\mathcal{X}^s}&:=&\sup_{(q,r)\in \Lambda _b}\left\|u\right\|_{L^qB^s_{r,2}\cap B^{s/4}_{q,2}L^r},\notag\\
	\left\|u\right\|_{L_{\text{uloc},T}^qL^p}&:=&\sup_{b-a=2T}\left(\int_{a}^{b}\left\|u\right\| _{L^p(\R^N)}^q\mathrm{d} t\ \right)^{1/q},\notag
\end{eqnarray}
where $s>0$, $1\le p,q<\infty $ and $T>0$.
In the rest of this paper, we fix the cut off function $\chi\in  C_0^\infty((-2,2))$ with $\chi|_{t\in[-1,1]}=1$, $\chi_T(t)=\chi(\fc tT)$. We first consider the case $s<\frac{N}{2}$.
\begin{lemma}\label{l1}
	Let $N\ge1$, $\beta>\max \left\{2,\frac{N}{4}\right\} $, $0<s\le2$, $s<\frac{N}{2}$,  $0<\alpha$,  $\left(N-2s\right)\alpha <8-\frac{2N}{\beta}$, $f\in\mathcal{C}(\alpha )$, and $K(x)\in L^1(\R^N)\cap L^\beta(\R^N)$.  There exist $\left(\gamma,\rho\right)\in\Lambda_b$ and $\sigma>0$ such that
	\begin{eqnarray}\label{2231}
		&&\left(\int_{-\infty }^{\infty }\left(\left|\tau\right|^{-s/4 }\left\|\left(\chi_T(t-\tau)-\chi_T(\tau)\right)K(x)\left(f(u)-f(v)\right)\right\|_{L^{\gamma'}L^{\rho'}}\right)^2 \frac{\mathrm{d}\tau}{\left|\tau\right|} \right)^{1/2}\notag\\
		&\lesssim &  T^\sigma \left(\left\|u\right\|_{\mathcal{X}^s}^\alpha +\left\|v\right\|_{\mathcal{X}^s}^\alpha \right)\left\|u-v\right\|_{\mathcal{X}^s}.
	\end{eqnarray}
\end{lemma}
\begin{proof}
	Let $b=\frac{N}{\beta}$, then we have $0<b<\min \left\{\frac{N}{2},4\right\} $ and $(N-2s)\alpha <8-2b$.	We first claim that if $\left(\gamma,\rho\right),\left(q,r\right)\in\Lambda_b$ are two admissible pairs that satisfy
	\begin{numcases}
		{\ }	1-\frac{1}{\gamma}-\frac{\alpha+1 }{q}-\frac{s}{4}>0,\label{a1}\\
		r<\frac{N}{s},\label{a2}\\
		1-\frac{1}{\rho}-\left(\alpha +1\right)\left(\frac{1}{r}-\frac{s}{N}\right)>\frac{b}{N},\label{a3}
	\end{numcases}
	then the inequality (\ref{2231}) holds with $\sigma=1-\frac{1}{\gamma}-\frac{\alpha +1}{q}-\frac{s}{4}$.
	In fact, let $p,l$ be given by $1-\frac{1}{\rho}=\frac{1}{p}+\left(\alpha +1\right)\left(\frac{1}{r}-\frac{s}{N}\right)$ and $1-\frac{1}{\gamma}=\frac{\alpha+1 }{q}+\frac{1}{l}$, respectively. Then it is easy to check that $1<l<\infty $ and $1<p<\frac{N}{b}=\beta$; so that $K(x)\in L^1(\R^N)\cap L^\beta(\R^N)\subset L^p(\R^N)$. From (\ref{fu}), H\"older's inequality and Sobolev's embedding $B^{s}_{r,2}\left(\R^N\right)\hookrightarrow L^{\frac{Nr}{N-sr}}\left(\R^N\right)$, we have
	\begin{eqnarray}\label{292}
		&&\left\|\left(\chi_{T}(t-\tau)-\chi_T(t)\right)K\left(f(u)-f(v)\right)\right\|_{L^{\gamma'}L^{\rho'}}\notag\\
		&\lesssim &  \left\|\chi_{T}(t-\tau)-\chi_T(t)\right\|_{L^l}\left\|K\right\|_{L^p}\left(\left\|u\right\|_{L^qL^{\frac{Nr}{N-sr}}}^\alpha +\left\|v\right\|_{L^qL^{\frac{Nr}{N-sr}}}^\alpha \right)\left\|u-v\right\|_{L^qL^{\frac{Nr}{N-sr}}}\notag\\
		&\lesssim &  \left\|\chi_{T}(t-\tau)-\chi_T(t)\right\|_{L^l}\left(\left\|u\right\|_{L^qB^{s}_{r,2}}^\alpha +\left\|v\right\|_{L^qB^{s}_{r,2}}^\alpha \right)\left\|u-v\right\|_{L^qB^{s}_{r,2}}.
	\end{eqnarray}
	Moreover, from (\ref{292}), we have
	\begin{eqnarray}\label{293}
		&&\left(\int_{-\infty }^{\infty }\left(\left|\tau\right|^{-s/4 }\left\|\left(\chi_T(t-\tau)-\chi_T(\tau)\right)K(x)\left(f(u)-f(v)\right)\right\|_{L^{\gamma'}L^{\rho'}}\right)^2 \frac{\mathrm{d}\tau}{\left|\tau\right|} \right)^{1/2}\notag\\
		&\lesssim & \left\|\chi_T\right\|_{B^{s/4}_{l,2}}\left(\left\|u\right\|_{L^qB^{s}_{r,2}}^\alpha +\left\|v\right\|_{L^qB^{s}_{r,2}}^\alpha \right)\left\|u-v\right\|_{L^qB^{s}_{r,2}}.
	\end{eqnarray}
	The inequality (\ref{2231}) is now an immediate consequence of (\ref{293}) and 	 $\left\|\chi_T\right\|_{B^{s/4}_{l,2}}\lesssim C_{\chi}T^{\frac{1}{l}-\frac{s}{4}}=C_{\chi}T^{1-\frac{1}{\gamma}-\frac{\alpha+1 }{q}-\frac{s}{4}}$.
	
	To prove Lemma \ref{l1}, it suffices to  provide two  biharmonic admissible pairs $\left(\gamma,\rho\right),\left(q,r\right)\in\Lambda_b$ that satisfy (\ref{a1})--(\ref{a3}). We consider four cases.

	\textbf{Cases 1:} $\left(N-2s\right)\alpha <2s-2b$. Let $\gamma=q=\infty$,  $\rho=r=2$. Then it is easy to check that $\left(\gamma,\rho\right)$, $\left(q,r\right)\in\Lambda_b$ and (\ref{a1}), (\ref{a2}) hold. For (\ref{a3}), we have
	\begin{equation}
		1-\frac{1}{\rho}-\left(\alpha +1\right)\left(\frac{1}{r}-\frac{s}{N}\right)-\frac{b}{N}=\frac{2s-2b-\left(N-2s\right)\alpha }{2N}>0.\notag
	\end{equation}
	
	\textbf{Cases 2:} $2s-2b\le\left(N-2s\right)\alpha<4+2s-2b$. Let $q=\infty$, $r=2$ and
	\begin{equation}
		\gamma=\frac{8}{\left(N-2s\right)\alpha -2s+2b+\ep},\qquad  \rho=\frac{2N}{N+2s-2b-\left(N-2s\right)\alpha -\ep},\notag
	\end{equation}
	where $\ep>0$ sufficiently small such that
	\begin{equation}
		\ep<\min \left\{4+2s-2b-\left(N-2s\right)\alpha,\  8-2b-\left(N-2s\right)\alpha, \ 2  \right\}. \notag
	\end{equation}
	Then it is easy to check that $\left(\gamma,\rho\right)$, $\left(q,r\right)\in\Lambda_b$ and (\ref{a2}) hold. Moreover, by direct computation, we have
	\begin{equation}
		\begin{cases}
			1-\frac{1}{\gamma}-\frac{\alpha+1 }{q}-\frac{s}{4}=\frac{8-2b-\left(N-2s\right)\alpha -\ep}{8}>0,\\
			1-\frac{1}{\rho}-\left(\alpha +1\right)\left(\frac{1}{r}-\frac{s}{N}\right)=\frac{2b+\ep}{2N}>\frac{b}{N}.
		\end{cases}\notag
	\end{equation}
	Hence we see that (\ref{a1}) and (\ref{a3}) hold.

	\textbf{Cases 3:} $\left(N-2s\right)\alpha \ge 4+2s-2b$ and $N\ge5$. Let $\gamma=2$, $\rho=\frac{2N}{N-4}$, and
	\begin{equation}
		q=\frac{8\left(\alpha +1\right)}{\left(N-2s\right)\alpha -\left(4+2s-2b\right)+2\ep},\qquad  r=\frac{2N\left(\alpha +1\right)}{N+2s\left(\alpha +1\right)+4-2b-2\ep},\notag
	\end{equation}
	where $\ep>0$ sufficiently small such that
	\begin{equation}
		\ep<\min \left\{\frac{8-2b-\left(N-2s\right)\alpha}{2},\ s\right\}. \notag
	\end{equation}
	Then it is easy to check that  and
	Then  by direct calculation, we have $\left(\gamma,\rho\right)$, $\left(q,r\right)\in\Lambda_b$ and
	\begin{equation}
		\begin{cases}
			1-\frac{1}{\gamma}-\frac{\alpha +1}{q}-\frac{s}{4}=\frac{8-2b-\left(N-2s\right)\alpha -2\ep}{8}>0,\\
			\frac{N}{s}-r=\frac{N\left(N+4-2b-2\ep\right)}{s\left(N+2s\left(\alpha +1\right)+4-2b-2\ep\right)}>0,\notag\\
			1-\frac{1}{\rho}-\left(\alpha +1\right)\left(\frac{1}{r}-\frac{s}{N}\right)=\frac{b+\ep}{N}>\frac{b}{N}.
		\end{cases}\notag
	\end{equation}
	Hence we see that (\ref{a1})--(\ref{a3}) hold.
	
	\textbf{Cases 4:} $\left(N-2s\right)\alpha \ge 4+2s-2b$ and $N\le4$. Let $\gamma=\frac{8}{N(1-\ep)}$, $\rho=\frac{2}{\ep}$, and
	\begin{equation}
		q=\frac{8(\alpha +1)}{(N-2s)\alpha-(1-2\ep)N-2s+2b},\qquad r=\frac{N(\alpha +1)}{(1-\ep)N-b+s(\alpha +1)},\notag
	\end{equation}
	where $\ep>0$ sufficiently small such that
	\begin{equation}
		\ep< \min \left\{\frac{N-b}{N},\ \frac{8-2b-(N-2s)\alpha }{N}\right\}.    \notag
	\end{equation}
	Then by direct calculation, we have  $\left(\gamma,\rho\right)$, $ \left(q,r\right)\in\Lambda_b$  and
	\begin{equation}
		\begin{cases}
			1-\frac{1}{\gamma}-\frac{\alpha +1}{q}-\frac{s}{4}=\frac{8-2b-\ep N-\left(N-2s\right)\alpha}{8}>0,\\
			\frac{N}{s}-r=\frac{N\left((1-\ep)N-b\right)}{s\left((1-\ep)N-b+s(\alpha +1)\right)}>0,\\
			1-\frac{1}{\rho}-\left(\alpha +1\right)\left(\frac{1}{r}-\frac{s}{N}\right)=\frac{2b+N\ep}{2N}>\frac{b}{N}.
		\end{cases}\notag
	\end{equation}
	Hence we see that (\ref{a1})--(\ref{a3}) hold.
\end{proof}
\begin{lemma}\label{l2}
	Let $N\ge1$, $\beta>\max \left\{2,\frac{N}{4}\right\} $, $0<s\le2$, $s<\frac{N}{2}$,  $0<\alpha$,  $\left(N-2s\right)\alpha <8-\frac{2N}{\beta}$, $f\in\mathcal{C}(\alpha )$, and $K(x)\in L^\infty (\R^N)$.  There exist $\left(\gamma,\rho\right)\in\Lambda_b$ and $\sigma>0$ such that
	\begin{eqnarray}\label{2232}
		&&\left(\int_{-\infty }^{\infty }\left(\left|\tau\right|^{-s/4 }\left\|\left(\chi_T(t-\tau)-\chi_T(\tau)\right)K(x)\left(f(u)-f(v)\right)\right\|_{L^{\gamma'}L^{\rho'}}\right)^2 \frac{\mathrm{d}\tau}{\left|\tau\right|} \right)^{1/2}\notag\\
		&\lesssim &  T^\sigma \left(\left\|u\right\|_{\mathcal{X}^s}^\alpha +\left\|v\right\|_{\mathcal{X}^s}^\alpha \right)\left\|u-v\right\|_{\mathcal{X}^s}.
	\end{eqnarray}
\end{lemma}
\begin{proof}
	We first claim that if $\left(\gamma,\rho\right)$, $\left(q,r\right)\in\Lambda_b$ are two admissible pairs that satisfy
	\begin{numcases}
		{\ }	1-\frac{1}{\gamma}-\frac{\alpha+1 }{q}-\frac{s}{4}>0,\label{b1}\\
		r<\frac{N}{s},\label{b2}\\
		\frac{\alpha +1}{r}>	1-\frac{1}{\rho}>\left(\alpha +1\right)\left(\frac{1}{r}-\frac{s}{N}\right),\label{b3}
	\end{numcases}
	then the inequality (\ref{2232}) holds with $\sigma=1-\frac{1}{\gamma}-\frac{\alpha +1}{q}-\frac{s}{4}$.
	In fact, let $p,l$ be given by $1-\frac{1}{\rho}=\frac{\alpha +1}{p}$ and $1-\frac{1}{\gamma}=\frac{\alpha+1 }{q}+\frac{1}{l}$, respectively.  Then by (\ref{b3}) we have  $\frac{1}{r}-\frac{s}{N}<\frac{1}{p}<\frac{1}{r}$; so that  the embedding $B^{s}_{r,2}\left(\R^N\right)\hookrightarrow L^p(\R^N)$ holds.  Similar to (\ref{293}), we  deduce from H\"older's inequality and Sobolev's embedding $B^{s}_{r,2}\left(\R^N\right)\hookrightarrow L^p(\R^N)$ that
	\begin{eqnarray}\label{295}
		&&\left(\int_{-\infty }^{\infty }\left(\left|\tau\right|^{-s/4 }\left\|\left(\chi_T(t-\tau)-\chi_T(\tau)\right)K(x)\left(f(u)-f(v)\right)\right\|_{L^{\gamma'}L^{\rho'}}\right)^2 \frac{\mathrm{d}\tau}{\left|\tau\right|} \right)^{1/2}\notag\\
		&\lesssim & \left\|\chi_T\right\|_{B^{s/4}_{l,2}}\left(\left\|u\right\|_{L^qB^{s}_{r,2}}^\alpha +\left\|v\right\|_{L^qB^{s}_{r,2}}^\alpha \right)\left\|u-v\right\|_{L^qB^{s}_{r,2}},
	\end{eqnarray} where we also used the boundedness of $K(x)$. 	
	The inequality (\ref{2232}) is now an immediate consequence of (\ref{295}) and $\left\|\chi_T\right\|_{B^{s/4}_{l,2}}\lesssim C_{\chi}T^{\frac{1}{l}-\frac{s}{4}}=C_{\chi}T^{1-\frac{1}{\gamma}-\frac{\alpha+1 }{q}-\frac{s}{4}}$.

	To  prove Lemma \ref{l2}, it suffices to  provide two biharmonic admissible pairs $\left(\gamma,\rho\right),\left(q,r\right)\in\Lambda_b$ that satisfy (\ref{b1})--(\ref{b3}). We consider four cases.
	
	\textbf{Cases 1:} $\left(N-2s\right)\alpha <2s$. Let $\gamma=q=\infty$, $\rho=r=2$.
	
	\textbf{Cases 2:} $2s\le\left(N-2s\right)\alpha<4+2s$. Let $q=\infty$, $r=2$ and
	\begin{equation}
		\gamma=\frac{8}{\left(N-2s\right)\alpha -2s+\ep},\qquad  \rho=\frac{2N}{N+2s-\left(N-2s\right)\alpha -\ep},\notag
	\end{equation}
	where $\ep>0$ sufficiently small such that
	\begin{equation}
		\ep<\min \left\{4+2s-\left(N-2s\right)\alpha, \ 8-\left(N-2s\right)\alpha, \ \frac{2s(\alpha +1)}{N}  \right\}. \notag
	\end{equation}
	
	\textbf{Cases 3:} $\left(N-2s\right)\alpha \ge 4+2s$ and $N\ge5$. Let $\gamma=2$, $\rho=\frac{2N}{N-4}$, and
	\begin{equation}
		q=\frac{8\left(\alpha +1\right)}{\left(N-2s\right)\alpha -\left(4+2s\right)+2\ep},\qquad  r=\frac{2N\left(\alpha +1\right)}{N+2s\left(\alpha +1\right)+4-2\ep},\notag
	\end{equation}
	where $\ep>0$ sufficiently small such that
	\begin{equation}
		\ep<\min \left\{\frac{8-\left(N-2s\right)\alpha }{2},\ s\right\}. \notag
	\end{equation}

	\textbf{Cases 4:} $\left(N-2s\right)\alpha \ge 4+2s$ and $N\le4$. Let $\gamma=\frac{8}{N(1-\ep)}$, $\rho=\frac{2}{\ep}$, and
	\begin{equation}
		q=\frac{8(\alpha +1)}{(N-2s)\alpha-(1-2\ep)N-2s},\qquad r=\frac{N(\alpha +1)}{(1-\ep)N+s(\alpha +1)},\notag
	\end{equation}
	where $\ep>0$ sufficiently small such that
	\begin{equation}
		\ep< \min \left\{\frac{1}{2},\ \frac{8-(N-2s)\alpha }{N}\right\}.    \notag
	\end{equation}
	
	In each case, we can  verify that  $\left(\gamma,\rho\right)$, $\left(q,r\right)\in\Lambda_b$,  (\ref{b1})-(\ref{b3}) hold and  omit the details.
\end{proof}
\begin{lemma}\label{l3}
	Let $N\ge1$, $\beta>\max \left\{2,\frac{N}{4}\right\} $, $0<s\le2$, $s<\frac{N}{2}$,  $0<\alpha$,  $\left(N-2s\right)\alpha <8-\frac{2N}{\beta}$, $f\in\mathcal{C}(\alpha )$, and $K(x)\in L^1(\R^N)\cap L^\beta(\R^N)$.   There exist $\sigma>0$ and $\widetilde{q}>1$, $\left(\gamma,\rho\right)$, $\left(q,r\right)\in\Lambda_b$ with $\widetilde{q}<q$, $r<\min \left\{\frac{N}{s},4^*\right\} $ such that
	\begin{align}\label{2233}
		&\left(\int_{-\infty }^{\infty }\left(\left|\tau\right|^{-s/4 }\left\|\chi_T(t-\tau)K(x)\left(\left(f(u)-f(v)\right)_\tau-\left(f(u)-f(v)\right)\right)\right\|_{L^{\gamma'}L^{\rho'}}\right)^2 \frac{\mathrm{d}\tau}{\left|\tau\right|} \right)^{1/2}\notag\\
		&\lesssim   T^\sigma \left(\left\|u\right\|_{\mathcal{X}^s}^\alpha +\left\|v\right\|_{\mathcal{X}^s}^\alpha \right)\left\|u-v\right\|_{\mathcal{X}^s}+1_{\alpha <1}\left\|u-v\right\|_{L^{\widetilde{q}}_{\text{uloc},T}L^{\frac{Nr}{N-sr}}}^\alpha \left\|v\right\|_{\mathcal{X}^s}.
	\end{align}
\end{lemma}
\begin{proof}
	Let $b=\frac{N}{\beta}$, then we have $0<b<\min \left\{\frac{N}{2},4\right\} $ and $(N-2s)\alpha <8-2b$.
	We first claim that if $\left(\gamma,\rho\right)$, $\left(q,r\right)$, $\left(m,n\right)\in\Lambda_b$ are three admissible pairs that satisfy
	\begin{numcases}
		{\ }	1-\frac{1}{\gamma}-\frac{\alpha }{q}-\frac{1}{m}>0,\label{c1}\\
		r<\frac{N}{s},\label{c2}\\
		1-\frac{1}{\rho}-\alpha \left(\frac{1}{r}-\frac{s}{N}\right)-\frac{1}{n}>\frac{b}{N}.\label{c3}
	\end{numcases}
	Then we can find $1<\widetilde{q}<q$ such that the inequality (\ref{2233}) holds with $\sigma=1-\frac{1}{\gamma}-\frac{\alpha }{q}-\frac{1}{m}$.
	In fact, from
	\begin{eqnarray}\label{10105}
		&&\left(f\left(u_{\tau}\right)-f\left(v_{\tau}\right)\right)-\left(f(u)-f(v)\right) \notag\\
		&=&\left(\left(u_{\tau}-v_{\tau}\right)-\left(u-v\right)\right) \int_{0}^{1} f^{\prime}\left(u+\theta\left(u_{\tau}-u\right)\right) d \theta \notag\\
		&&+\left(v_{\tau}-v\right) \int_{0}^{1}\left[f^{\prime}\left(u+\theta\left(u_{\tau}-u\right)\right)-f^{\prime}\left(v+\theta\left(v_{\tau}-v\right)\right)\right] d \theta\notag \\
		&=&A_{1}+A_{2},
	\end{eqnarray}
	we have,  by applying (\ref{fu})
	\begin{equation}\label{10106}
		\left|A_{1}\right| \lesssim\left|\left(u_{\tau}-v_{\tau}\right)-\left(u-v\right)\right|\left(\left|u\right|^{\alpha}+\left|u_{\tau}\right|^{\alpha}\right),
	\end{equation}
	and
	\begin{equation}\label{10107}
		\left|A_{2}\right| \leq\left\{\begin{array}{ll}			\left|v_{\tau}-v\right|\left(\left|u\right|+\left|u_{\tau}\right|+|v|+\left|v_{\tau}\right|\right)^{\alpha-1}\left(\left|u-v\right|
			+\left|u_{\tau}-v_{\tau}\right|\right), &
			\text {if } \alpha \geqslant 1, \\
			\left|v_{\tau}-v\right|\left(\left|u-v\right|^{\alpha}+\left|u_{\tau}-v_{\tau}\right|^{\alpha}\right), &
			\text {if } 0<\alpha<1.
		\end{array}\right.
	\end{equation}
	Put $1-\frac{1}{\rho}=\frac{1}{p}+\alpha \left(\frac{1}{r}-\frac{s}{N}\right)+\frac{1}{n}$. Then by (\ref{c3}), we have $1<p<\frac{N}{b}=\beta$ and thus $K(x)\in L^1(\R^N)\cap L^\beta(\R^N) \subset L^p(\R^N)$.
	From (\ref{10105}),  (\ref{10106}) and (\ref{10107}), the boundedness of $\chi_T$ and H\"older's inequality, we get
	\begin{eqnarray}
		&&\left\|\chi_T(t-\tau)K(x)\left(\left(f(u)-f(v)\right)_\tau-\left(f(u)-f(v)\right)\right)\right\|_{L^{\gamma'}L^{\rho'}}\notag\\
		&\lesssim&  T^{\sigma}\|(u-v)_\tau-(u-v)\|_{L^m L^n}\left\|K(x)\right\|_{L^p}\left(\left\|u\right\|_{L^qL^{\frac{Nr}{N-sr}}}^\alpha+\left\|v\right\|_{L^qL^{\frac{Nr}{N-sr}}}^\alpha \right)\notag\\
		&&+\begin{cases}T^{\sigma}\left\|K(x)\right\|_{L^p}\|v_\tau-v\|_{L^m L^n}\left(\left\|u\right\|_{L^qL^{\frac{Nr}{N-sr}}}^{\alpha -1}+\left\|v\right\|_{L^qL^{\frac{Nr}{N-sr}}}^{\alpha -1}\right)\left\|u-v\right\|_{L^qL^{\frac{Nr}{N-sr}}},\ \alpha\ge1,\\
			\|v_\tau-v\|_{L^m L^n}\left\|K(x)\right\|_{L^p}\|u-v\|_{L^{\widetilde{q}}_{\text{uloc},T} L^{\frac{Nr}{N-sr}}}^\alpha,\  0<\alpha<1,
		\end{cases}\notag
	\end{eqnarray}
	where $\sigma=1-\frac{1}{\gamma}-\frac{\alpha }{q}-\frac{1}{m}$ and $1<\widetilde{q}<q$ is given by $1-\frac{1}{\gamma}=\frac{\alpha }{\widetilde{q}}+\frac{1}{m}$.
	This inequality together with Sobolev's embedding $B^{s}_{r,2}\left(\R^N\right)\hookrightarrow L^{\frac{Nr}{N-sr}}\left(\R^N\right)$ implies
	\begin{align}\label{2103}
		&\left(\int_{-\infty }^{\infty }\left(\left|\tau\right|^{-s/4 }\left\|\chi_T(t-\tau)K(x)\left(\left(f(u)-f(v)\right)_\tau-\left(f(u)-f(v)\right)\right)\right\|_{L^{\gamma'}L^{\rho'}}\right)^2 \frac{\mathrm{d}\tau}{\left|\tau\right|} \right)^{1/2}\notag\\
		&\lesssim   T^{\sigma}\left\|u-v\right\|_{B^{s/4}_{m,2}L^n}\left(\left\|u\right\|_{L^qB^{s}_{r,2}}^\alpha +\left\|v\right\|_{L^qB^{s}_{r,2}}^\alpha \right)\notag\\
		&\qquad + \begin{cases}T^{\sigma}\|v\|_{B^{s/4}_{m,2}L^n}\left(\left\|u\right\|_{L^qB^{s}_{r,2}}^{\alpha -1}+\left\|v\right\|_{L^qB^{s}_{r,2}}^{\alpha -1}\right)\left\|u-v\right\|_{L^qB^{s}_{r,2}},\ \alpha\ge1,\\
			\|v\|_{B^{s/4}_{m,2}L^n}\|u-v\|_{L^{\widetilde{q}}_{\text{uloc},T} L^{\frac{Nr}{N-sr}}}^\alpha,\  0<\alpha<1.
		\end{cases}		
	\end{align}
	The inequality (\ref{2233}) is now an immediate consequence of (\ref{2103}) and Young's inequality.
	
	To prove Lemma \ref{l3}, it suffices to  provide three biharmonic admissible pairs $\left(\gamma,\rho\right)$,  $\left(q,r\right)$, $(m,n)\in\Lambda_b$ with $r<4^*$ that satisfy (\ref{c1})--(\ref{c3}). We consider two cases.
	
	\textbf{Case 1:} $N\ge 2s+4$. Let	
	\begin{equation}
		\begin{cases}
			\gamma=2,\qquad &\rho=\frac{2N}{N-4},\\
			q=m=\frac{4\left(\alpha +1\right)}{2-\ep},\qquad &r=n=\frac{2N\left(\alpha +1\right)}{N\left(\alpha +1\right)-4+2\ep},
		\end{cases}\notag
	\end{equation}
	where $\ep>0$ sufficiently small such that
	\begin{equation}
		\ep<\min \left\{\frac{8-2b-\left(N-2s\right)\alpha }{2},\ 2\right\} .\notag
	\end{equation}
	Then by direct calculation, we have  $\left(\gamma,\rho\right)$, $\left(q,r\right)\in\Lambda_b$ with $r<4^*$, and
	\begin{equation}
		\begin{cases}
			1-\frac{1}{\gamma}-\frac{\alpha }{q}-\frac{1}{m}=\frac{\ep}{4}>0,\\
			\frac{N}{s}-r=\frac{N\left(\left(N-2s\right)\alpha +N-2s-4+2\ep\right)}{s\left(N\left(\alpha +1\right)-4+2\ep\right)}>0,\\
			1-\frac{1}{\rho}-\alpha \left(\frac{1}{r}-\frac{s}{N}\right)-\frac{1}{n}-\frac{b}{N}=\frac{8-2b-\left(N-2s\right)\alpha-2\ep }{2N}>0.
		\end{cases}		\notag
	\end{equation}
	Hence we see that (\ref{c1})--(\ref{c3}) hold.
	
	\textbf{Cases 2:} $N<2s+4$. Let
	\begin{equation}
		\begin{cases}
			\gamma=m=\frac{8}{b+2\alpha\ep},\qquad  &\rho=n=\frac{2N}{N-b-2\alpha\ep},\\
			q=\frac{8}{N-2s-2\ep},\qquad  &r=\frac{N}{s+\ep},
		\end{cases}	\notag
	\end{equation}
	where $\ep>0$ sufficiently small such that
	\begin{equation}
		\ep<\min \left\{\frac{N-b}{2\alpha },\ \frac{8-2b-(N-2s)\alpha }{2 \alpha }, \ \frac{N-2s}{2}\right\}. \notag
	\end{equation}
	Then it is easy to check that $\left(q,r\right)$, $\left(\gamma,\rho\right)\in\Lambda_b$ with $r<\frac{N}{s}<4^*$. Moreover, by direct computation, we have
	\begin{equation}
		\begin{cases}
			1-\frac{1}{\gamma}-\frac{\alpha }{q}-\frac{1}{m}=\frac{8-2b-\left(N-2s\right)\alpha-2\ep \alpha  }{8}>0,\\
			1-\frac{1}{\rho}-\alpha  \left(\frac{1}{r}-\frac{s}{N}\right)-\frac{1}{n}-\frac{b}{N}=\frac{\alpha \ep}{N}>0.
		\end{cases}\notag
	\end{equation}
	Hence we see that (\ref{c1})--(\ref{c3}) hold.
\end{proof}
\begin{lemma}\label{l4}
	Let $N\ge1$, $\beta>\max \left\{2,\frac{N}{4}\right\} $, $0<s\le2$, $s<\frac{N}{2}$,  $0<\alpha$,  $\left(N-2s\right)\alpha <8-\frac{2N}{\beta}$, $f\in\mathcal{C}(\alpha )$, and $K(x)\in L^\infty (\R^N)$.  There exist $\sigma>0$ and $\widetilde{q}>1$,  $\left(\gamma,\rho\right)$, $\left(q,r\right)\in\Lambda_b$ with $\widetilde{q}<q$, $r<\min \left\{\frac{N}{s},4^*\right\} $ such that
	\begin{align}\label{2234}
		&\left(\int_{-\infty }^{\infty }\left(\left|\tau\right|^{-s/4 }\left\|\chi_T(t-\tau)K(x)\left(\left(f(u)-f(v)\right)_\tau-\left(f(u)-f(v)\right)\right)\right\|_{L^{\gamma'}L^{\rho'}}\right)^2 \frac{\mathrm{d}\tau}{\left|\tau\right|} \right)^{1/2}\notag\\
		&\lesssim   T^\sigma \left(\left\|u\right\|_{\mathcal{X}^s}^\alpha +\left\|v\right\|_{\mathcal{X}^s}^\alpha \right)\left\|u-v\right\|_{\mathcal{X}^s}+1_{\alpha <1}\|u-v\|_{L^{\widetilde{q}}_{\text{uloc},T} L^{\frac{Nr}{N-sr}}}^\alpha \left\|v\right\|_{\mathcal{X}^s}.
	\end{align}
\end{lemma}
\begin{proof}
	Let
	\begin{equation}
		q=\frac{8(\alpha +2)}{(N-2s)\alpha +2\ep},\qquad r=\frac{N(\alpha +2)}{N+\alpha s-\ep},\notag
	\end{equation}
	where $\ep>0$ sufficiently small such that
	\begin{equation}
		\ep<\min \left\{N-2s,\ \frac{8-(N-2s)\alpha }{2},\ \alpha s\right\}.
	\end{equation}
	Then by direct calculation, we have $r<4^*$ and
	\begin{equation}
		\begin{cases}
			1-\frac{\alpha +2}{q}=\frac{8-(N-2s)\alpha -2\ep}{8}>0,\\
			\frac{N}{s}-r=\frac{N(N-2s-\ep)}{s(N+\alpha s-\ep)}>0.
		\end{cases}\notag
	\end{equation}
	Let $p,\sigma$ be given by $1=\frac{2}{r}+\frac{\alpha }{p}$ and $1=\frac{1}{\sigma}+\frac{\alpha +2}{q}$, respectively.  Since
	\begin{equation}
		\begin{cases}
			1-\frac{2}{r}-\frac{\alpha }{r}=-\frac{\alpha s-\ep}{N}<0,\\
			1-\frac{2}{r}-\alpha \left(\frac{1}{r}-\frac{s}{N}\right)=\frac{\ep}{N}>0,
		\end{cases}\notag
	\end{equation}
	we have $\frac{1}{r}-\frac{s}{N}<\frac{1}{p}<\frac{1}{r}$; so that the embedding $B^s_{r,2}(\R^N)\hookrightarrow L^p(\R^N)$ holds.
	Similar to (\ref{2103}), we  deduce from H\"older's inequality and Sobolev's embedding $B^{s}_{r,2}\left(\R^N\right)\hookrightarrow L^p(\R^N)$ that
	\begin{align}\label{2107}
		&\left(\int_{-\infty }^{\infty }\left(\left|\tau\right|^{-s/4 }\left\|\chi_T(t-\tau)K(x)\left(\left(f(u)-f(v)\right)_\tau-\left(f(u)-f(v)\right)\right)\right\|_{L^{q'}L^{r'}}\right)^2 \frac{\mathrm{d}\tau}{\left|\tau\right|} \right)^{1/2}\notag\\
		&\lesssim   T^{\sigma}\left\|u-v\right\|_{B^{s/4}_{q,2}L^r}\left(\left\|u\right\|_{L^qB^{s}_{r,2}}^\alpha +\left\|v\right\|_{L^qB^{s}_{r,2}}^\alpha \right)\notag\\
		&\qquad + \begin{cases}T^{\sigma}\|v\|_{B^{s/4}_{q,2}L^r}\left(\left\|u\right\|_{L^qB^{s}_{r,2}}^{\alpha -1}+\left\|v\right\|_{L^qB^{s}_{r,2}}^{\alpha -1}\right)\left\|u-v\right\|_{L^qB^{s}_{r,2}},&\alpha\ge1,\\
			\|v\|_{B^{s/4}_{q,2}L^r}\|u-v\|_{L^{\widetilde{q}}_{\text{uloc},T} L^{\frac{Nr}{N-sr}}}^\alpha,&0<\alpha<1,
		\end{cases}		
	\end{align}
	where $1<\widetilde{q}<q$ is given by $1=\frac{2}{q}+\frac{\alpha }{\widetilde{q}}$.
	The inequality (\ref{2234}) is now an immediate consequence of (\ref{2107}) and Young's inequality.
\end{proof}
\begin{lemma}\label{l5}
	Let $N\ge1$, $\beta>\max \left\{2,\frac{N}{4}\right\} $, $0<s\le2$, $s<\frac{N}{2}$,  $0<\alpha$,  $\left(N-2s\right)\alpha <8-\frac{2N}{\beta}$, $f\in\mathcal{C}(\alpha )$, and $K(x)\in L^1(\R^N)\cap L^\beta(\R^N)$.  There exists  $\sigma>0$ and $1\le \overline{q}$, $\overline{r}\le2$ with $\frac{4}{\overline{q}}-N\left(\frac{1}{2}-\frac{1}{\overline{r}}\right)=4-s$ such that
	\begin{equation}
		\left\|K(x)\chi_T(t)\left(f(u)-f(v)\right)\right\|_{L^{\overline{q}}L^{\overline{r}}}\lesssim T^\theta\left(\left\|u\right\|_{\mathcal{X}^s}^\alpha+\left\|v\right\|_{\mathcal{X}^s}^\alpha\right)\left\|u-v\right\|_{\mathcal{X}^s}.\label{2235}
	\end{equation}
\end{lemma}
\begin{proof}
	Let $b=\frac{N}{\beta}$, then we have $0<b<\min \left\{\frac{N}{2},4\right\} $ and $(N-2s)\alpha <8-2b$.	We first claim that if $1\le \overline{q},\overline{r}\le2$ with $\frac{4}{\overline{q}}-N\left(\frac{1}{2}-\frac{1}{\overline{r}}\right)=4-s$ and  $\left(q,r\right)\in \Lambda _b$ is an admissible pair that satisfy
	\begin{numcases}
		{\ }\frac{1}{\overline{q}}-\frac{\alpha +1}{q}>0,\label{e1}\\
		r<\frac{N}{s},\label{e2}\\
		\frac{1}{\overline{r}}-(\alpha +1)\left(\frac{1}{r}-\frac{s}{N}\right)>\frac{b}{N},\label{e3}
	\end{numcases}
	then the inequality (\ref{2235}) holds with $\sigma=\frac{1}{\overline{q}}-\frac{\alpha +1}{q}$. In fact, let $1<p<\infty $ be given by $\frac{1}{\overline{r}}=\frac{1}{p}+(\alpha +1)\left(\frac{1}{r}-\frac{s}{N}\right)$, then by (\ref{e3}), we have $1<p<\frac{N}{b}=\beta$; so that $K(x)\in L^1(\R^N)\cap L^\beta (\R^N)\subset L^p(\R^N)$. Using (\ref{fu}), H\"older's inequality and Sobolev's embedding $B^s_{r,2}\left(\R^N\right) \hookrightarrow  L^{\frac{Nr}{N-sr}\left(\R^N\right)}$, we have
	\begin{eqnarray}
		&&\left\|K(x)\chi_T(t)\left(f(u)-f(v)\right)\right\|_{L^{\overline{q}}L^{\overline{r}}}\notag\\
		&\lesssim &  \left\|\left\|K\right\|_{L^p}\left(\left\|u\right\|_{L^{\frac{Nr}{N-sr}}}^\alpha +\left\|v\right\|_{L^{\frac{Nr}{N-sr}}}^\alpha \right)\left\|u-v\right\|_{L^{\frac{Nr}{N-sr}}}\right\|_{L^{\overline{q}}}\notag\\
		&\lesssim &  T^{\frac{1}{\overline{q}}-\frac{\alpha +1}{q}}\left(\left\|u\right\|_{L^{q}B^{s}_{r,2}}^\alpha+\left\|v\right\|_{L^{q}B^{s}_{r,2}}^\alpha\right)\left\|u-v\right\|_{L^{q}B^{s}_{r,2}},\notag
	\end{eqnarray}
	which yields (\ref{2235}).
	
	To prove Lemma \ref{l5}, it sufficies to provide $1\le \overline{q},\overline{r}\le2$ with $\frac{4}{\overline{q}}-N\left(\frac{1}{2}-\frac{1}{\overline{r}}\right)=4-s$ and  an admissible pair $(q,r)\in\Lambda_b$ that satisfy (\ref{e1})--(\ref{e3}). We consider two cases.

	\textbf{Case 1:} $N+2s\ge4$.  Let $\overline{q}=2$, $\overline{r}=\frac{2N}{N+4-2s}$, then we have $1\le \overline{q}$, $\overline{r}\le2$ and $\frac{4}{\overline{q}}-N\left(\frac{1}{2}-\frac{1}{\overline{r}}\right)=4-s$. Next, we choose the admissible pair $(q,r)\in\Lambda_b$ that satisfy (\ref{e1})--(\ref{e3}). We consider two subcases.  \\
	\textbf{Subcase 1:} $\left(N-2s\right)\alpha <4-2b$.
	Let
	\begin{equation}
		q=\infty ,\qquad r=2.\notag
	\end{equation}
	Then it is easy to verify that $(q,r)\in\Lambda_b$ and  (\ref{e1})--(\ref{e3}) hold. \\
	\textbf{Subcase 2:} $\left(N-2s\right)\alpha \ge 4-2b$. Let
	$$
	q=\frac{8\left(\alpha +1\right)}{\left(N-2s\right)\alpha -4+2b+2\ep},\qquad  r=\frac{2N\left(\alpha +1\right)}{N+4-2b-2\ep+2s\alpha},\notag
	$$
	where $\ep>0$ sufficiently small such that
	\begin{equation}
		\ep<\min \left\{\frac{8-2b-\left(N-2s\right)\alpha}{2} ,\ 2\alpha , \  \frac{N+4-2b-2s}{2}\right\}. \notag
	\end{equation}
	Then we have $(q,r)\in\Lambda_b$ and
	\begin{equation}
		\begin{cases}
			\frac{1}{\overline{q}}-\frac{\alpha +1}{q}=\frac{8-2b-2\ep-(N-2s)\alpha }{8}>0,\\
			\frac{N}{s}-r=\frac{N\left(N+4-2b-2s-2\ep\right)}{s\left(N+4-2b-2\ep+2s\alpha \right)}>0,\\
			\frac{1}{\overline{r}}-(\alpha +1)\left(\frac{1}{r}-\frac{s}{N}\right)=\frac{b+\ep}{N}>\frac{b}{N}.
		\end{cases}\notag
	\end{equation}
	Hence we see that (\ref{e1})--(\ref{e3}) hold.
	
	\textbf{Case 2:} $N+2s<4$. Let $\overline{q}=\frac{4}{4-s}$, $\overline{r}=2$, then we have $1\le \overline{q},\overline{r}\le2$ and $\frac{4}{\overline{q}}-N\left(\frac{1}{2}-\frac{1}{\overline{r}}\right)=4-s$. Next, we choose the admissible pair $(q,r)\in\Lambda_b$ that satisfy (\ref{e1})--(\ref{e3}). We consider two subcases.\\
	\textbf{Subcase 1:} $\left(N-2s\right)\alpha <2s-2b$.
	Let
	\begin{equation}
		q=\infty ,\qquad r=2.\notag
	\end{equation}
	Then it is easy to verify that $(q,r)\in\Lambda_b$ and  (\ref{e1})--(\ref{e3}) hold. \\
	\textbf{Subcase 2:} $\left(N-2s\right)\alpha \ge 2s-2b$. Let
	\begin{equation}
		q=\frac{8\left(\alpha +1\right)}{\left(N-2s\right)\alpha -2s+2b+2\ep},\qquad  r=\frac{2N\left(\alpha +1\right)}{N-2b-2\ep+2s(\alpha+1)},\notag
	\end{equation}
	where $\ep>0$ sufficiently small such that
	\begin{equation}
		2\ep<\min \left\{8-2b-\left(N-2s\right)\alpha,\ N-2b\right\}. \notag
	\end{equation}
	Then we have $(q,r)\in\Lambda_b$ and
	\begin{equation}
		\begin{cases}
			\frac{1}{\overline{q}}-\frac{\alpha +1}{q}=\frac{8-2b-2\ep-(N-2s)\alpha }{8}>0,\\
			\frac{N}{s}-r=\frac{N(N-2b-2\ep)}{s(N-2b-2\ep+2s(\alpha +1))}>0,\\
			\frac{1}{\overline{r}}-(\alpha +1)\left(\frac{1}{r}-\frac{s}{N}\right)=\frac{b+\ep}{N}>\frac{b}{N}.
		\end{cases}\notag
	\end{equation}
	Hence  we have that (\ref{e1})--(\ref{e3}) hold.
\end{proof}
\begin{lemma}\label{l6}
	Let $N\ge1$, $\beta>\max \left\{2,\frac{N}{4}\right\} $, $0<s\le2$, $s<\frac{N}{2}$,  $0<\alpha$,  $\left(N-2s\right)\alpha <8-\frac{2N}{\beta}$, $f\in\mathcal{C}(\alpha )$, and $K(x)\in L^\infty (\R^N)$.  There exist $\sigma>0$, $1\le \overline{q},\overline{r}\le2$ with $\frac{4}{\overline{q}}-N\left(\frac{1}{2}-\frac{1}{\overline{r}}\right)=4-s$ such that
	\begin{equation}\label{2236}
		\left\|K\chi_T\left(f(u)-f(v)\right)\right\|_{L^{\overline{q}}L^{\overline{r}}} \lesssim  T^{\sigma}\left(\left\|u\right\|_{\mathcal{X}^s}^\alpha+\left\|v\right\|_{\mathcal{X}^s}^\alpha\right)\left\|u-v\right\|_{\mathcal{X}^s}.
	\end{equation}
\end{lemma}
\begin{proof}
	We first claim that if $1\le \overline{q},\overline{r}\le2$ with $\frac{4}{\overline{q}}-N\left(\frac{1}{2}-\frac{1}{\overline{r}}\right)=4-s$ and  $\left(q,r\right)\in \Lambda _b$ is an admissible pair that satisfy
	\begin{numcases}
		{\ }\frac{1}{\overline{q}}-\frac{\alpha +1}{q}>0,\label{f1}\\
		r<\frac{N}{s},\label{f2}\\
		\frac{\alpha +1}{r}>\frac{1}{\overline{r}}>(\alpha +1)\left(\frac{1}{r}-\frac{s}{N}\right),\label{f3}
	\end{numcases}
	then the inequality (\ref{2236}) holds with $\sigma=\frac{1}{\overline{q}}-\frac{\alpha +1}{q}$. In fact, let $1<p<\infty $ be given by $\frac{1}{\overline{r}}=\frac{\alpha +1}{p}$, then by (\ref{f3}), we have $\frac{1}{r}-\frac{s}{N}<\frac{1}{p}<\frac{1}{r}$; so that  the embedding $B^s_{r,2}(\R^N)\hookrightarrow L^{p}(\R^N)$ holds. Using (\ref{fu}), the boundedness of $K(x)$, H\"older's inequality and Sobolev's embedding $B^s_{r,2}\left(\R^N\right) \hookrightarrow  L^{\frac{Nr}{N-sr}\left(\R^N\right)}$, we have
	\begin{eqnarray}
		&&\left\|K(x)\chi_T(t)\left(f(u)-f(v)\right)\right\|_{L^{\overline{q}}L^{\overline{r}}}\notag\\
		&\lesssim &  \left\|\left\|K\right\|_{L^p}\left(\left\|u\right\|_{L^{\frac{Nr}{N-sr}}}^\alpha +\left\|v\right\|_{L^{\frac{Nr}{N-sr}}}^\alpha \right)\left\|u-v\right\|_{L^{\frac{Nr}{N-sr}}}\right\|_{L^{\overline{q}}}\notag\\
		&\lesssim &  T^{\frac{1}{\overline{q}}-\frac{\alpha +1}{q}}\left(\left\|u\right\|_{L^{q}B^{s}_{r,2}}^\alpha+\left\|v\right\|_{L^{q}B^{s}_{r,2}}^\alpha\right)\left\|u-v\right\|_{L^{q}B^{s}_{r,2}},\notag
	\end{eqnarray}
	which yields (\ref{2236}).
	
	To prove Lemma \ref{l6}, it suffices to provide $1\le \overline{q},\overline{r}\le2$ with $\frac{4}{\overline{q}}-N\left(\frac{1}{2}-\frac{1}{\overline{r}}\right)=4-s$ and an  $(q,r)\in\Lambda_b$ that satisfy (\ref{f1})--(\ref{f3}). We consider three cases.
	
	\textbf{Case 1:} $(N-2s)\alpha <2s$.  Let
	\begin{equation}
		\begin{cases}
			\overline{q}=\frac{4}{4-s},\qquad &\overline{r}=2,\\
			q=\infty ,\qquad &r=2.
		\end{cases}\notag
	\end{equation}
	Then it is easy to verify that $1\le \overline{q},\overline{r}\le2$ with $\frac{4}{\overline{q}}-N\left(\frac{1}{2}-\frac{1}{\overline{r}}\right)=4-s$ and an admissible pair $(q,r)\in\Lambda_b$ that satisfies (\ref{f1})--(\ref{f3}).

	\textbf{Case 2:} $2s\le (N-2s)\alpha < 4\alpha +4+2s$.
	Let
	\begin{equation}
		\begin{cases}
			\overline{q}=\frac{4}{4-s},\qquad &\overline{r}=2,\\
			q=\frac{8\left(\alpha +1\right)}{\left(N-2s\right)\alpha -2s+2\ep},\qquad  &r=\frac{2N\left(\alpha +1\right)}{N-2\ep+2s(\alpha+1)},\notag
		\end{cases}
	\end{equation}
	where $\ep>0$ sufficiently small such that
	\begin{equation}
		\ep<\min \left\{\frac{8-\left(N-2s\right)\alpha}{2} ,\  s(\alpha +1),\  \frac{N}{2},\ \frac{4\alpha +4-2s-(N-2s)\alpha }{2} \right\}. \notag
	\end{equation}
	By direct calculation, we have $1\le \overline{q},\overline{r}\le2$ with $\frac{4}{\overline{q}}-N\left(\frac{1}{2}-\frac{1}{\overline{r}}\right)=4-s$, $(q,r)\in\Lambda_b$
	and
	\begin{equation}
		\begin{cases}
			\frac{1}{\overline{q}}-\frac{\alpha +1}{q}=\frac{8-(N-2s)\alpha -2\ep}{8}>0,\\
			\frac{N}{s}-r=\frac{N(N-2\ep)}{s(N-2\ep+2s(\alpha +1))}>0,\\
			\frac{\alpha +1}{r}-\frac{1}{\overline{r}}=\frac{s(\alpha +1)-\ep}{N}>0,\\
			\frac{1}{\overline{r}}-(\alpha +1)\left(\frac{1}{r}-\frac{s}{N}\right)=\frac{\ep}{N}>0,
		\end{cases}\notag
	\end{equation}
	which imply  (\ref{f1})--(\ref{f3}).
	
	\textbf{Case 3:} $(N-2s)\alpha >4\alpha +4+2s$.
	Let
	\begin{equation}
		\begin{cases}
			\overline{q}=\frac{8(\alpha +2)}{(8-(N-2s))\alpha +16-2\ep},\qquad &\overline{r}=\frac{N(\alpha +2)}{(N-2s)\alpha +N-2s+\ep},\\
			q=\frac{8(\alpha +2)}{(N-2s)\alpha +2\ep},\qquad &r=\frac{N(\alpha +2)}{N+s\alpha -\ep},\notag
		\end{cases}
	\end{equation}
	where $\ep>0$ sufficiently small  such that
	\begin{equation}
		\ep<\min \left\{b,\ s,\ N-2s, \ \frac{8_(N-2s)\alpha }{2}\right\}.\notag
	\end{equation}
	By direct calculation, we have $1\le \overline{q},\overline{r}\le2$ with $\frac{4}{\overline{q}}-N\left(\frac{1}{2}-\frac{1}{\overline{r}}\right)=4-s$, $(q,r)\in\Lambda_b$
	and
	\begin{equation}
		\begin{cases}
			\frac{1}{\overline{q}}-\frac{\alpha +1}{q}=\frac{\left(8-2\ep-(N-2s)\alpha \right)(\alpha +2)}{8(\alpha +2)}>0,\\
			\frac{N}{s}-r=\frac{N(N-2s-\ep)}{s(N+s\alpha +\ep)}>0,\\
			\frac{\alpha +1}{r}-\frac{1}{\overline{r}}=\frac{(\alpha +2)s\alpha +s+\ep}{N(\alpha +2)}>0,\\
			\frac{1}{\overline{r}}-(\alpha +1)\left(\frac{1}{r}-\frac{s}{N}\right)=\frac{\ep \alpha }{N(\alpha +2)}>0.
		\end{cases}\notag
	\end{equation}
	Hence we see that (\ref{f1})--(\ref{f3}) hold.
\end{proof}
\begin{lemma}\label{l7}
	Let $N\ge1$, $\beta>\max \left\{2,\frac{N}{4}\right\} $, $0<s\le2$, $s<\frac{N}{2}$,  $0<\alpha$,  $\left(N-2s\right)\alpha <8-\frac{2N}{\beta}$, $f\in\mathcal{C}(\alpha )$, and $K(x)\in L^1(\R^N)\cap L^\beta(\R^N)$.   There exist $\sigma>0$, $(\gamma,\rho)\in\Lambda_b $ such that
	\begin{equation}\label{12236}
		\left\|K\chi_T\left(f(u)-f(v)\right)\right\|_{L^{\gamma'}L^{\rho'}}\lesssim T^{\sigma}\left(\left\|u\right\|_{\mathcal{X}^s}^\alpha+\left\|v\right\|_{\mathcal{X}^s}^\alpha\right)\sup_{(q,r)\in\Lambda_b}\left\|u-v\right\|_{L^qL^r}.
	\end{equation}
\end{lemma}
\begin{proof}
	Let $b=\frac{N}{\beta}$, then we have $0<b<\min \left\{\frac{N}{2},4\right\} $ and $(N-2s)\alpha <8-2b$.	From the proof of Lemma \ref{l3}, we know there exist three admissible pairs $(\gamma,\rho),$ $(q,r),(m,n)\in\Lambda_b$ such that
	\begin{numcases}
		{\ }	1-\frac{1}{\gamma}-\frac{\alpha }{q}-\frac{1}{m}>0,\notag\\
		r<\frac{N}{s},\notag\\
		1-\frac{1}{\rho}-\alpha \left(\frac{1}{r}-\frac{s}{N}\right)-\frac{1}{n}>\frac{b}{N}.\notag
	\end{numcases}
	Put $\frac{1}{p}=1-\frac{1}{\rho}-\alpha \left(\frac{1}{r}-\frac{s}{N}\right)-\frac{1}{n}$, then we have $1<p<\frac{N}{b}$ and thus $K(x)\in L^p(\R^N)$. Using (\ref{fu}), H\"older's inequality and Sobolev's embedding $B^s_{r,2}(\R^N)\hookrightarrow L^{\frac{Nr}{N-sr}}(\R^N)$, we have
	\begin{eqnarray}
		&&\left\|K\chi_T\left(f(u)-f(v)\right)\right\|_{L^{\gamma'}L^{\rho'}}\notag\\
		&\lesssim& T^{1-\frac{1}{\gamma}-\frac{\alpha }{q}-\frac{1}{m}} \left\|K\right\|_{L^p}\left(\left\|u\right\|_{L^qL^{\frac{Nr}{N-sr}}}^\alpha +\left\|v\right\|_{L^qL^{\frac{Nr}{N-sr}}}^\alpha \right)\left\|u-v\right\|_{L^mL^n}\notag\\
		&\lesssim &  T^{1-\frac{1}{\gamma}-\frac{\alpha }{q}-\frac{1}{m}} \left(\left\|u\right\|_{L^qB^s_{r,2}}^\alpha +\left\|v\right\|_{L^qB^s_{r,2}}^\alpha \right)\left\|u-v\right\|_{L^mL^n},\notag
	\end{eqnarray}
	which yields (\ref{12236}).
\end{proof}

\begin{lemma}\label{l8}
	Let $N\ge1$, $\beta>\max \left\{2,\frac{N}{4}\right\} $, $0<s\le2$, $s<\frac{N}{2}$,  $0<\alpha$,  $\left(N-2s\right)\alpha <8-\frac{2N}{\beta}$, $f\in\mathcal{C}(\alpha )$, and $K(x)\in L^\infty (\R^N)$.   There exist $\sigma>0$, $(\gamma,\rho)\in\Lambda_b $ such that
	\begin{equation}\label{2237}
		\left\|K\chi_T\left(f(u)-f(v)\right)\right\|_{L^{\gamma'}L^{\rho'}}\lesssim T^{\sigma}\left(\left\|u\right\|_{\mathcal{X}^s}^\alpha+\left\|v\right\|_{\mathcal{X}^s}^\alpha\right)\sup_{(q,r)\in\Lambda_b}\left\|u-v\right\|_{L^qL^r}.
	\end{equation}
\end{lemma}
\begin{proof}
	From the proof of Lemma \ref{l4}, we know there exists $(q,r)\in \Lambda_b$ such that
	\begin{numcases}
		{\ }1-\frac{\alpha +2}{q}>0,\notag\\
		r<\frac{N}{s},\notag\\
		\frac{\alpha }{r}+\frac{1}{r}>1-\frac{1}{r}>\alpha \left(\frac{1}{r}-\frac{s}{N}\right)+\frac{1}{r}.\notag
	\end{numcases}
	Let $p$ be given by $1-\frac{1}{r}=\frac{\alpha }{p}+\frac{1}{r}$, then we have $\frac{1}{r}-\frac{s}{N}<\frac{1}{p}<\frac{1}{r}$ and thus the embedding $B^s_{r,2}(\R^N)\hookrightarrow L^p(\R^N)$ holds. Using (\ref{fu}), the boundedness of $K(x)$, H\"older's inequality and Sobolev's embedding $B^s_{r,2}(\R^N)\hookrightarrow L^{p}(\R^N)$, we have
	$$
	\left\|K\chi_T\left(f(u)-f(v)\right)\right\|_{L^{q'}L^{r'}}
	\lesssim   T^{1-\frac{\alpha +2}{q}} \left(\left\|u\right\|_{L^qB^s_{r,2}}^\alpha +\left\|v\right\|_{L^qB^s_{r,2}}^\alpha \right)\left\|u-v\right\|_{L^qL^r},\notag
	$$
	which yields (\ref{2237}).
\end{proof}
In the case $s\ge \frac{N}{2}$, the embedding $H^s(\R^N)\hookrightarrow L^p(\R^N)$ holds for any $2\le p<\infty $; thus we can establish the following Lemmas.
\begin{lemma}\label{l9}
	Let $N\ge1$, $\beta>\max \left\{2,\frac{N}{4}\right\} $, $0<s\le2$, $s\ge\frac{N}{2}$,  $0<\alpha<\infty $, $f\in\mathcal{C}(\alpha )$, and $K(x)\in L^\infty (\R^N)+ L^\beta(\R^N)$.  There exist $\sigma>0$ and $2\le r,\rho<\infty $ such that
	\begin{align}\label{z1}
		&\left(\int_{-\infty }^{\infty }\left(\left|\tau\right|^{-s/4 }\left\|\left(\chi_T(t-\tau)-\chi_T(\tau)\right)K(x)\left(f(u)-f(v)\right)\right\|_{L^1L^2}\right)^2 \frac{\mathrm{d}\tau}{\left|\tau\right|} \right)^{1/2}\notag\\
		&\lesssim   T^\sigma \left(\left\|u\right\|_{\mathcal{X}^s}^\alpha +\left\|v\right\|_{\mathcal{X}^s}^\alpha \right)\left\|u-v\right\|_{\mathcal{X}^s},
	\end{align}
	and
	\begin{align}\label{z2}
		&\left(\int_{-\infty }^{\infty }\left(\left|\tau\right|^{-s/4 }\left\|\chi_T(t-\tau)K(x)\left(\left(f(u)-f(v)\right)_\tau-\left(f(u)-f(v)\right)\right)\right\|_{L^1L^2}\right)^2 \frac{\mathrm{d}\tau}{\left|\tau\right|} \right)^{1/2}\notag\\
		&\lesssim   T^\sigma \left(\left\|u\right\|_{\mathcal{X}^s}^\alpha +\left\|v\right\|_{\mathcal{X}^s}^\alpha \right)\left\|u-v\right\|_{\mathcal{X}^s}+1_{\alpha <1}\left\|u-v\right\|_{L^\infty \left(\R,L^r\cap L^\rho\right)}^\alpha \left\|v\right\|_{\mathcal{X}^s}.
	\end{align}
\end{lemma}
\begin{proof}
	Since $K(x)\in L^\infty (\R^N)+ L^\beta(\R^N)\subset L^\infty (\R^N)+L^1(\R^N)\cap L^\beta (\R^N)$, we can decompose $K(x)=K_1(x)+K_2(x)$, with $K_1(x)\in L^\infty (\R^N)$, $K_2(x)\in L^1(\R^N)\cap L^\beta(\R^N)$. Put $(q,r)$, $(\gamma,\rho)\in\Lambda_b$,
	\begin{equation}
		\begin{cases}
			q=\frac{8(\alpha +1)(2+\ep)}{N(2\alpha +2+\ep \alpha )},\qquad &r=\frac{2(2+\ep)(\alpha +1)}{\ep},\\
			\gamma=\frac{8(\alpha +1)}{N\alpha },\qquad &\rho=2\alpha +2,
		\end{cases}\notag
	\end{equation}
	where $0<\ep<\beta-2$; so that $K_2(x)\in L^1(\R^N)\cap L^\beta(\R^N)\subset L^{2+\ep}(\R^N)$.
	
	We first prove (\ref{z2}). 	From (\ref{10105}),  (\ref{10106}) and (\ref{10107}), the boundedness of $\chi_T$ and H\"older's inequality, we get
	\begin{eqnarray}
		&&\left\|\chi_T(t-\tau)K_2(x)\left(\left(f(u)-f(v)\right)_\tau-\left(f(u)-f(v)\right)\right)\right\|_{L^1L^2}\notag\\
		&\lesssim&  T^{1-\frac{1}{q}}\|(u-v)_\tau-(u-v)\|_{L^q L^r}\left\|K_2(x)\right\|_{L^{2+\ep}}\left(\left\|u\right\|_{L^\infty L^{r}}^\alpha+\left\|v\right\|_{L^\infty L^{r}}^\alpha \right)\notag\\
		&&+T^{1-\frac{1}{q}}\begin{cases}\|v_\tau-v\|_{L^q L^r}\left(\left\|u\right\|_{L^\infty L^{r}}^{\alpha -1}+\left\|v\right\|_{L^\infty L^{r}}^{\alpha -1}\right)\left\|u-v\right\|_{L^qL^{r}}, \ \alpha\ge1,\\
			\|v_\tau-v\|_{L^q L^r}\|u-v\|_{L^{\infty } L^{r}}^\alpha,\ 0<\alpha<1.
		\end{cases}\notag
	\end{eqnarray}
	This inequality together with young's inequality yields
	\begin{align}\label{z21}
		&\left(\int_{-\infty }^{\infty }\left(\left|\tau\right|^{-s/4 }\left\|\chi_T(t-\tau)K_2(x)\left(\left(f(u)-f(v)\right)_\tau-\left(f(u)-f(v)\right)\right)\right\|_{L^1L^2}\right)^2 \frac{\mathrm{d}\tau}{\left|\tau\right|} \right)^{1/2}\notag\\
		&\lesssim   T^{1-\frac{1}{q}} \left(\left\|u\right\|_{\mathcal{X}^s}^\alpha +\left\|v\right\|_{\mathcal{X}^s}^\alpha \right)\left\|u-v\right\|_{\mathcal{X}^s}+1_{\alpha <1}\left\|u-v\right\|_{L^\infty L^r}^\alpha \left\|v\right\|_{\mathcal{X}^s}.
	\end{align}
	Similarly, we have
	\begin{align}\label{z22}
		&\left(\int_{-\infty }^{\infty }\left(\left|\tau\right|^{-s/4 }\left\|\chi_T(t-\tau)K_1(x)\left(\left(f(u)-f(v)\right)_\tau-\left(f(u)-f(v)\right)\right)\right\|_{L^1L^2}\right)^2 \frac{\mathrm{d}\tau}{\left|\tau\right|} \right)^{1/2}\notag\\
		&\lesssim   T \left(\left\|u\right\|_{\mathcal{X}^s}^\alpha +\left\|v\right\|_{\mathcal{X}^s}^\alpha \right)\left\|u-v\right\|_{\mathcal{X}^s}+1_{\alpha <1}\left\|u-v\right\|_{L^\infty L^\rho}^\alpha \left\|v\right\|_{\mathcal{X}^s}.
	\end{align}
	The inequality (\ref{z2}) is now an immediate consequence of (\ref{z21}) and (\ref{z22}).
	
	Next, we prove (\ref{z1}). From (\ref{fu}) and H\"older's inequality, we have
	\begin{eqnarray}
		&&\left\|\left(\chi_{T}(t-\tau)-\chi_T(t)\right)K_2\left(f(u)-f(v)\right)\right\|_{L^1L^2}\notag\\
		&\lesssim &  \left\|\chi_{T}(t-\tau)-\chi_T(t)\right\|_{L^1}\left\|K_2\right\|_{L^{2+\ep}}\left(\left\|u\right\|_{L^\infty L^{r}}^\alpha +\left\|v\right\|_{L^\infty L^{r}}^\alpha \right)\left\|u-v\right\|_{L^\infty L^{r}}.\notag
	\end{eqnarray}
	This inequality together with Sobolev's embedding $H^s(\R^N)\hookrightarrow L^{r}(\R^N)$ yields
	\begin{align}\label{z11}
		&\left(\int_{-\infty }^{\infty }\left(\left|\tau\right|^{-s/4 }\left\|\left(\chi_T(t-\tau)-\chi_T(\tau)\right)K_2(x)\left(f(u)-f(v)\right)\right\|_{L^1L^2}\right)^2 \frac{\mathrm{d}\tau}{\left|\tau\right|} \right)^{1/2}\notag\\
		&\lesssim   T^{1-\frac{s}{4}} \left(\left\|u\right\|_{\mathcal{X}^s}^\alpha +\left\|v\right\|_{\mathcal{X}^s}^\alpha \right)\left\|u-v\right\|_{\mathcal{X}^s},
	\end{align}
	where we used $\left\|\chi_T\right\|_{B^{s/4}_{1,2}}\lesssim T^{1-\frac{s}{4}}$. Similarly, we have
	\begin{align}\label{z12}
		&\left(\int_{-\infty }^{\infty }\left(\left|\tau\right|^{-s/4 }\left\|\left(\chi_T(t-\tau)-\chi_T(\tau)\right)K_1(x)\left(f(u)-f(v)\right)\right\|_{L^1L^2}\right)^2 \frac{\mathrm{d}\tau}{\left|\tau\right|} \right)^{1/2}\notag\\
		&\lesssim   T \left(\left\|u\right\|_{\mathcal{X}^s}^\alpha +\left\|v\right\|_{\mathcal{X}^s}^\alpha \right)\left\|u-v\right\|_{\mathcal{X}^s}.
	\end{align}
	The inequality (\ref{z1}) is now an immediate consequence of (\ref{z11}) and (\ref{z12}).
\end{proof}
Similar to Lemma \ref{l9}, we can establish the following Lemma and omit the details.
\begin{lemma}\label{l10}
	Let $N\ge1$, $\beta>\max \left\{2,\frac{N}{4}\right\} $, $0<s\le2$, $s\ge\frac{N}{2}$,  $0<\alpha<\infty $, $f\in\mathcal{C}(\alpha )$, and $K(x)\in L^\infty (\R^N)+ L^\beta(\R^N)$. There exist $\sigma>0$  such that
	\begin{equation}\label{x1}
		\left\|K\chi_T\left(f(u)-f(v)\right)\right\|_{L^{\overline{q}}L^{\overline{r}}} \lesssim  T^{\sigma}\left(\left\|u\right\|_{\mathcal{X}^s}^\alpha+\left\|v\right\|_{\mathcal{X}^s}^\alpha\right)\left\|u-v\right\|_{\mathcal{X}^s},\notag
	\end{equation}
	where $\overline{q}=\frac{4}{4-s}$, $\overline{r}=2$ with $1\le \overline{q},\overline{r}\le2$, $\frac{4}{\overline{q}}-N\left(\frac{1}{2}-\frac{1}{\overline{r}}\right)=4-s$ and
	\begin{equation}\label{x2}
		\left\|K\chi_T\left(f(u)-f(v)\right)\right\|_{L^1L^2}\lesssim T^{\sigma}\left(\left\|u\right\|_{\mathcal{X}^s}^\alpha+\left\|v\right\|_{\mathcal{X}^s}^\alpha\right)\sup_{(q,r)\in\Lambda_b}\left\|u-v\right\|_{L^qL^r}.\notag
	\end{equation}
\end{lemma}
\section{Proof of Theorem \ref{T1} and  Theorem \ref{T2}}\label{5}
\subsection{\textbf{The local existence}}
Since $K(x)\in L^\infty (\R^N)+ L^\beta(\R^N)\subset L^\infty (\R^N)+L^1(\R^N)\cap L^\beta (\R^N)$, we can decompose $K(x)=K_1(x)+K_2(x)$, with $K_1(x)\in L^\infty (\R^N)$, $K_2(x)\in L^1(\R^N)\cap L^\beta(\R^N)$. Then using Lemmas \ref{l1}--\ref{l10} and the standard contraction mapping argument, we can   establish the following local existence results easily and omit the details.
\begin{proposition}\label{914}
	Let $N\ge1$, $\mu=-1$ or $0$, $\beta>\max \left\{2,\frac{N}{4}\right\} $, $0<s\le2$, $0<\alpha $,  $\left(N-2s\right)\alpha <8-\frac{2N}{\beta}$, $f\in\mathcal{C}(\alpha )$, and $K(x)\in L^\infty (\R^N)+ L^\beta(\R^N)$.
	There exist  constants $C_1>0$ and $c(L)>0$ that depends only on $L$ such that for any $L\ge 2C_1\left\|\varphi\right\|_{H^s}, 0<T\le c(L)$, the equation
	\begin{equation}\label{911}
		u=e^{i t (\Delta^2+\mu\Delta)} \varphi-i\int_{0}^{t}e^{i(t-\tau)(\Delta^2+\mu\Delta)}\left(\chi_TKf(u)\right)(s)  \mathrm{d}\tau
	\end{equation}
	admits  a unique solution $ u \in C (\R, H^s(\R^N)) \cap\mathcal{X}^s$ with $\|u\|_{\mathcal{X}^s} \leq L$.
\end{proposition}
Since $\chi_T|_{t\in[0,T]}=1$, it follows from   Proposition \ref{914} that there exists a solution $u\in C([0,T],H^s) \bigcap_{(q,r)\in\Lambda_b}L^{q}\left([0,T], B_{r, 2}^{s}\right)$ to the Cauchy problem (\ref{NLS}).
Moreover, using  the similar  arguments as that in Chapter 4 of \cite{Ca9},
we  deduce that  the Cauchy problem (\ref{NLS}) admits  a unique  maximum solution
\[
u\in C([0,T_{\text{max}}(\varphi)),H^s)\bigcap_{(q,r)\in \Lambda _b}L^{q}\left([0,T_{\text{max}}(\varphi)), B_{r, 2}^{s}\right)
\]
with  the blowup alternative (\ref{bl1-1}) holds, and omit the details.

\subsection{\textbf{The continuous dependence}}
In  this subsection,  we prove  continuous dependence of the solution map on the interval $[0,A]$, where $0<A<T_{\text{max}}(\varphi)$.
Let
$$L=4C_1\left\|u\right\|_{L^\infty \left([0,T],H^s\right)}<\infty,$$
and $T>0$ sufficiently small such that
\begin{equation}\label{9171}
	T\le c(L), \qquad C_2T^{\widetilde{\sigma}}L^\alpha +C_3T^\sigma L^\alpha \le \frac{1}{4},
\end{equation}
where $C_1,c(L)$ are the constants  in Proposition \ref{914},  $C_2, \widetilde{\sigma}$ and $C_3,\sigma$ are the constants in (\ref{1076}) and (\ref{1071}), respectively. Since   $\varphi_n\rightarrow\varphi$ in $H^s$, we know that there exists a positive $n_1$ such that $L\ge2C_1\|\varphi_n\|_{H^s}$ for every $n\ge n_1$. It then follows from Proposition \ref{914} that for every $n\ge n_1$, the following equation
\begin{equation*}
	u_n=e^{i t (\Delta^2+\mu\Delta)} \varphi_n-i\int_{0}^{t}e^{i(t-\tau)(\Delta^2+\mu\Delta)}\left(\chi_TKf(u_n)\right)(\tau)  \mathrm{d}\tau,
\end{equation*}
admits a unique solution $u_n\in C(\R,H^s)\cap\mathcal{X}^s$ with
\begin{equation}\label{M1}
	\|u_n\|_{\mathcal{X}^s} \leq L,\ \text{for } \forall\  n\ge n_1.
\end{equation}
Moreover, Proposition \ref{914} also gives us a unique solution $v\in C(\R,H^s)\cap \mathcal{X}^s$ to the equation (\ref{911}) that satisfies
\begin{equation}\label{9175}
	\|v\|_{\mathcal{X}^s} \leq L.
\end{equation}
The proof of the continuous dependence will proceed by the following claims.
\begin{claim}\label{10221}
	For any $0<\delta<s$, we have
	$\sup_{(q,r)\in\Lambda_b}\|u_{n}-v\|_{L^{q}(\R, B_{r,2}^{s-\delta})}\underset{n \rightarrow \infty}{\longrightarrow} 0.\label{10181}$.
\end{claim}
\begin{proof}
	Let $K(x)=K_1(x)+K_2(x)$ with $K_1(x)\in L^{\beta}(\R^N)$ and $K_2(x)\in L^\infty (\R^N)$.  From Lemmas \ref{l7},  \ref{l8} and \ref{l10}, we conclude that there exist $(\gamma_i,\rho_i)\in\Lambda_b$, $\sigma_i>0$, $i=1,2$, such that
	\begin{eqnarray}\label{2251}
		&&\left\|	K_i\chi_T\left(f(u)-f(v)\right)\right\|_{L^{\gamma_i'}L^{\rho_i'}}\notag\\
		&\lesssim& T^{\sigma_i}\left(\left\|u\right\|_{\mathcal{X}^s}^\alpha +\left\|v\right\|_{\mathcal{X}^s}^\alpha \right)\sup_{(q,r)\in\Lambda_b}\left\|u-v\right\|_{L^qL^r},\ i=1,2.
	\end{eqnarray}
	From Strichartz's estimate and (\ref{2251}), we deduce that there exists $C_2>0$ such that
	\begin{eqnarray}\label{1076}
		&&\underset{(q,r)\in\Lambda_b}\sup\|u_n-v\|_{L^qL^r}\notag\\
		&\le&C_2\left\|\varphi_{n}-\varphi\right\|_{L^{2}}+C_2T^{\widetilde{\sigma}}\left(\left\|u_n\right\|_{\mathcal{X}^s}^\alpha +\left\|v\right\|_{\mathcal{X}^s}^\alpha \right)\sup_{(q,r)\in\Lambda_b}\left\|u_n-v\right\|_{L^qL^r},
	\end{eqnarray}
	where $\widetilde{\sigma}=\min \left\{\sigma_1,\sigma_2\right\} $.
	Since  $2C_2T^{\widetilde{\sigma}}L^\alpha\le\fc12$ by (\ref{9171}), it follows from (\ref{M1}), (\ref{9175}) and (\ref{1076}) that
	\begin{equation}
		\sup_{(q,r)\in\Lambda_b}\left\|u_{n}-v\right\|_{L^{q}\left(\R, L^{r}\right)}\lesssim \|\varphi_n-\varphi\|_{L^2}.\notag
	\end{equation}
	Moreover,  from the interpolation theorem, (\ref{M1}) and (\ref{9175}),   we have
	\begin{eqnarray}
		\sup_{(q,r)\in\Lambda_b}\|u_{n}-v\|_{L^{q}(\R, B_{r,2}^{s-\delta})}&\lesssim &   \|u_{n}-v\|_{L^{q}(\R, B^s_{r,2})}^{1-\fc\delta s}\|u_n-v\|_{L^q(\R,L^r)}^{\fc\delta s} \notag\\
		&\lesssim &  L^{1-\fc\delta s}\|\varphi_{n}-\varphi\|_{L^{2}}^{\fc \delta s} \underset{n \rightarrow \infty}{\longrightarrow} 0,\notag
	\end{eqnarray}
	which finishes the proof of   Claim \ref{10221}.
\end{proof}
\begin{claim}\label{9176}
	$\|u_n-v\|_{\mathcal{X}^s}\underset{n \rightarrow \infty}{\longrightarrow} 0.$
\end{claim}
\begin{proof}
	When $s<\frac{N}{2}$, we deduce from Proposition \ref{p1}, Lemmas \ref{l1}--\ref{l8} that, for some $C_3,\sigma>0$,
	\begin{eqnarray}\label{1071}
		\|u_n-v\|_{\mathcal{X}^s}&\le & C_3  \|\varphi_n-\varphi\|_{H^s}+C_3T^\sigma \left(\left\|u_n\right\|_{\mathcal{X}^s}^\alpha +\left\|v\right\|_{\mathcal{X}^s}^\alpha \right)\left\|u_n-v\right\|_{\mathcal{X}^s}\notag\\
		&&+1_{\alpha <1}C_3\left\|u_n-v\right\|_{L^{\widetilde{q_1}}_{\text{uloc},T}L^{\frac{Nr_1}{N-sr_1}}\cap L^{\widetilde{q_2}}_{\text{uloc},T}L^{\frac{Nr_2}{N-sr_2}}}^\alpha \left\|v\right\|_{\mathcal{X}^s},
	\end{eqnarray}
	where $\widetilde{q_i}>1$, $(q_i,r_i)\in\Lambda_b$ with $\widetilde{q_i}<q_i$, $r_i<4^*,i=1,2$ are given by Lemmas \ref{l3} and \ref{l4},  respectively.
	Since $2C_3T^\sigma L^\alpha \le \frac{1}{2}$ by  (\ref{9171}), we deduce from (\ref{M1}), (\ref{9175}) and (\ref{1071}) that
	\begin{equation*}
		\|u_n-v\|_{\mathcal{X}^s}\lesssim\|\varphi_n-\varphi\|_{H^s}+\left\|u_n-v\right\|_{L^{\widetilde{q_1}}_{\text{uloc},T}L^{\frac{Nr_1}{N-sr_1}}\cap L^{\widetilde{q_2}}_{\text{uloc},T}L^{\frac{Nr_2}{N-sr_2}}}^\alpha.
	\end{equation*}
	Thus it suffices to  show that
	\begin{equation}\label{2261}
		\left\|u_n-v\right\|_{L^{\widetilde{q_1}}_{\text{uloc},T}L^{\frac{Nr_1}{N-sr_1}}\cap L^{\widetilde{q_2}}_{\text{uloc},T}L^{\frac{Nr_2}{N-sr_2}}} \underset{n \rightarrow \infty}{\longrightarrow} 0.
	\end{equation}

	For any  $\eta>0$ sufficiently small such that $r_1+\eta<4^*$, we choose  $q_\eta>0$ such that $(q_\eta,r_1+\eta)\in\Lambda_b$.   From Sobolev embedding $B^{s-\fc{\eta N}{r_1(r_1+\eta)}}_{r_1+\eta,2}(\R^N)\hookrightarrow L^{{\frac{Nr_1}{N-sr_1}}}(\R^N)$ and the claim \ref{10221}, we have
	\begin{equation}
		\left\|u_{n}-v\right\|_{L^{q_\eta}_{\text{uloc},T}L^{{\frac{Nr_1}{N-sr_1}}}} \lesssim \left\|u_{n}-v\right\|_{L^{q_\eta}B^{s-\fc{\eta N}{r_1(r_1+\eta)}}_{r_1+\eta,2}}  \underset{n \rightarrow \infty}{\longrightarrow} 0.
	\end{equation}
	Since $q_\eta$ increasing to $q_1$ as $\eta$ tends to $0$ and $\widetilde{q_1}<q_1$, we have, by applying H\"older's inequality
	\begin{equation}\label{2252}
		\|u_n-v\|_{L^{\widetilde{q_1}}_{\text{uloc},T}L^{{\frac{Nr_1}{N-sr_1}}}}\underset{n \rightarrow \infty}{\longrightarrow} 0.
	\end{equation}
	Similarly, we have
	\begin{equation}
		\|u_n-v\|_{L^{\widetilde{q_2}}_{\text{uloc},T}L^{{\frac{Nr_2}{N-sr_2}}}}\underset{n \rightarrow \infty}{\longrightarrow} 0. \notag
	\end{equation}
	This  together with (\ref{2252}) yields (\ref{2261}).
	
	We now consider the case $s\ge \frac{N}{2}$. In this case, we use Lemmas \ref{l9} and \ref{l10} instead of Lemmas \ref{l1}--\ref{l8}. Similar to (\ref{2261}), it suffices to prove that
	\begin{equation}
		\left\|u_n-v\right\|_{L^\infty \left(\R,L^r\cap L^\rho\right)}\underset{n \rightarrow \infty}{\longrightarrow} 0,\notag
	\end{equation}
	where $2\le r,\rho<\infty $ are given by Lemma \ref{l9}.
	Let $p=r+\rho$, then by interpolation theorem, there exists $0\le\theta<1$ such that
	\begin{equation}\label{323}
		\left\|u_n-v\right\|_{L^\infty L^r}\lesssim \left\|u_n-v\right\|_{L^\infty L^p}^\theta \left\|u_n-v\right\|_{L^\infty L^2}^{1-\theta }.
	\end{equation}
	From (\ref{323}), Sobolev's embedding $H^s(\R^N)\hookrightarrow L^p(\R^N)$,  (\ref{M1}), (\ref{9175}) and Claim \ref{10221}, we have
	\begin{equation}
		\left\|u_n-v\right\|_{L^\infty L^r}\underset{n \rightarrow \infty}{\longrightarrow} 0.\notag
	\end{equation}
	Similar, we have
	\begin{equation}
		\left\|u_n-v\right\|_{L^\infty L^\rho}\underset{n \rightarrow \infty}{\longrightarrow} 0.\notag
	\end{equation}
	This finishes the proof of Claim \ref{9176}.
\end{proof}

Now we resume the proof of the continuous dependence. Since  $\chi_T|_{t\in[0,T]}=1$, we  deduce from the uniqueness that $u(t)=v(t)$ on the interval $[0,T]$. Moreover, it follows from Claim \ref{9176} that $\|u_n-u\|_{L^\infty([0,T],H^s)}$ $\underset{n \rightarrow \infty}{\longrightarrow} 0$. In particular, we have $\|u_n(T)-u(T)\|_{H^s}\underset{n \rightarrow \infty}{\longrightarrow} 0.$ Arguing as previously, we deduce that the solution $u_n$ exists on the interval $[T,2T]$ for $n\ge n_2$ and that $\|u_n-u\|_{L^\infty([T,2T],H^s)}\underset{n \rightarrow \infty}{\longrightarrow} 0.$ Iterating finitely many times like this, we get the continuous dependence on the interval $[0,A]$.
\subsection{\textbf{Proof of Theorem \ref{T2}}}
Before proving Proposition \ref{p1},  we prove the following version of the Caffarelli-Kohn-Nirenberg-type inequality.
\begin{lemma}\label{10291}
	Let $N\ge1$, $\beta>\max \left\{2,\frac{N}{4}\right\} $,    $0<\alpha,\ (N-4)\alpha <8-\frac{2N}{\beta}$, $f\in\mathcal{C}(\alpha )$, and $K(x)\in L^\infty (\R^N)+ L^\beta(\R^N)$.  Then,  we have
	\begin{equation}\label{341}
		\int_{\mathbb{R}^{N}}\left|K(x)f(u)u\right| d x  \lesssim  \left\|\Delta u\right\|_{L^2}^{\frac{N\left(\beta\alpha +2\right)}{4\beta}}\left\|u\right\|_{L^2}^{\alpha +2-\frac{N\left(\beta\alpha +2\right)}{4\beta}}
		+\|\Delta u\|_{L^{2}}^{\frac{N \alpha}{4}}\|u\|_{L^{2}}^{\alpha+2-\frac{N \alpha}{4}}.
	\end{equation}
\end{lemma}
\begin{proof}
	Let $K(x)=K_1(x)+K_2(x)$ with $K_1(x)\in L^\infty (\R^N)$ and $K_2(x)\in L^{1}(\R^N)\cap L^\beta (\R^N)$.  From (\ref{fu}), H\"older's inequality and Gagliardo-Nirenberg's inequality, we have
	\begin{equation}\label{342}
		\int_{\mathbb{R}^{N}}\left|K_2(x)f(u)u\right| d x  \lesssim  \left\|K_2\right\|_{L^{\beta}}\left\|u\right\|_{L^{(\alpha +2)\beta'}}^{\alpha +2}
		\lesssim  \left\|\Delta u\right\|_{L^2}^{\frac{N\left(\beta\alpha +2\right)}{4\beta}}\left\|u\right\|_{L^2}^{\alpha +2-\frac{N\left(\beta\alpha +2\right)}{4\beta}}.
	\end{equation}
	Similarly, we have
	\begin{equation}\label{343}
		\int_{\mathbb{R}^{N}}\left|K_1(x)f(u)u\right| d x \lesssim  \|\Delta u\|_{L^{2}}^{\frac{N \alpha}{4}}\|u\|_{L^{2}}^{\alpha+2-\frac{N \alpha}{4}}.
	\end{equation}
	The inequality (\ref{341}) is now an immediate consequence of (\ref{342}) and (\ref{343}).
\end{proof}
Using the classical energy estimate method, we can obtain the following conservation law for the Cauchy problem (\ref{NLS}) easily and omit the details.
\begin{lemma}
	Let $N\ge1$, $\mu=-1$ or $0$,  $\beta>\max \left\{2,\frac{N}{4}\right\} $,  $0<\alpha,\ (N-4)\alpha  <8-\frac{2N}{\beta}$,   $K(x)\in L^\infty (\R^N)+L^\beta(\R^N)$ be a real-valued function and  $f\in\mathcal{C}(\alpha )$ that satisfies \\
	(i) $f(a)\in \R$ for all $a\ge0$; \\
	(ii) $f(u)=\frac{u}{\left|u\right|}f(\left|u\right|)$ for all $u\in\mathbb{C}\setminus \left\{0\right\} $;\\
	If $u$ is a smooth solution of (\ref{NLS}) on the time interval $[0,T]$, then the mass and the energy are conserved:
	\begin{equation}\label{m}
		M[u](t)=\int_{\R^N}|u|^2 dx=M[\varphi],
	\end{equation}
	\begin{equation}
		E(u)(t)=\int_{\R^N}\frac12\left(|\Delta u|^2-\mu|\nabla u|^2\right)\ dx+\int_{\R^N}K(x)G(u(x))\ dx=E[\phi],\label{e}
	\end{equation}
	where $G(z):=\int_{0}^{\left|z\right|}f(s) \mathrm{d}s$.
\end{lemma}
\begin{proof}[Proof of Theorem \ref{T2}]
	Let $u \in C \left(\left[0,T_{\text{max}}(\varphi)\right), H^2(\R^N)\right)$ be the  maximum solution given by Theorem \ref{T1}. To obtain a global solution, it is sufficient to get an a priori bound of the local solution, since the existence time obtained in Corollary \ref{Co1} depends only the $H^{2}$ norm of the initial datum.
	
	Note that  $\left|G(u(x))\right|\lesssim \left|f(u)u\right|$,  we deduce from   Gagliardo-Nirenberg's inequality, Lemma \ref{10291},  and the conservation of  the energy (\ref{e})	that, for any $t\in \left[0,T_{\text{max}}(\varphi)\right)$,
	\begin{eqnarray}\label{11125}
		\|\Delta u(t)\|_{L^{2}}^{2}
		&=&2 E\left[\phi\right]+\mu\|\nabla u(t)\|_2^2-2\int_{\R^N}K(x)G(u(x))\ dx \notag\\
		& \leq& 2 E\left[\phi\right]+\|u(t)\|_{L^2}\|\Delta u(t)\|_{L^2}+C\|\Delta u\|_{L^{2}}^{\frac{N \alpha}{4}}\|u\|_{L^{2}}^{\alpha+2-\frac{N \alpha}{4}}\notag\\
		&&+ \left\|\Delta u\right\|_{L^2}^{\frac{N\left(\beta\alpha +2\right)}{4\beta}}\left\|u\right\|_{L^2}^{\alpha +2-\frac{N\left(\beta\alpha +2\right)}{4\beta}}.
	\end{eqnarray}
	From  the conservation of the mass (\ref{m})	and  Cauchy-Schwartz's inequality,  we have,  for any $t\in \left[0,T_{\text{max}}(\varphi)\right)$,
	\begin{equation}\label{11126}
		\left\|u(t)\right\|_{L^2}\left\|\Delta u(t)\right\|_{L^2}
		\le  \fc14\|\Delta u(t)\|^2_{L^2}+\|\varphi\|_{L^2}^2.
	\end{equation}
	On the other hand,  from Young's inequality and the conservation of the mass  (\ref{m}), we have
	\begin{equation}
		\label{3111}C\|\Delta u(t)\|_{L^{2}}^{\frac{\alpha N}{4}}\left\|\varphi\right\|_{L^{2}}^{\alpha+2-\frac{\alpha N}{4}}\le \fc14\|\Delta u\|_{L^2}^2+C^2\|\varphi\|_{L^2}^{\left(\alpha+2-\frac{\alpha N}{4}\right)\fc8{8-\alpha N}}.
	\end{equation}

	Suppose $\alpha <\frac{8}{N}-\frac{2}{\beta}$. Since $\frac{N\left(\beta \alpha +2\right)}{4\beta}<2$, we have
	\begin{equation}\label{3112}
		C \left\|\Delta u\right\|_{L^2}^{\frac{N\left(\beta\alpha +2\right)}{4\beta}}\left\|u\right\|_{L^2}^{\alpha +2-\frac{N\left(\beta\alpha +2\right)}{4\beta}}
		\le    \frac{1}{4}\left\|\Delta u\right\|_{L^2}^2+C^2\left\|\phi\right\|_{L^2}^{\left(\alpha +2-\frac{N\left(\beta\alpha +2\right)}{4\beta}\right)\frac{8\beta}{(8-N\alpha )\beta-2N}}.
	\end{equation}
	From the above two inequalities,  (\ref{11125})--(\ref{3112}) and  the conservation of the mass (\ref{m}), we conclude that
	$\sup_{t\in \left[0,T_{\text{max}}(\varphi)\right) }\|u(t)\|_{H^2}<\infty$,
	which in turn implies the global existence.
	
	If $\alpha =\frac{8}{N}-\frac{2}{\beta}$, we deduce from (\ref{11125}), (\ref{11126}) and (\ref{3111}) that
	\begin{equation}
		\left(\frac{1}{2}-C\left\|\varphi\right\|_{L^2}^\alpha \right)\left\|\Delta u(t)\right\|_{L^2}^2\le 2E[\phi]+\|\varphi\|_{L^2}^2+\|\varphi\|_{L^2}^{\left(\alpha+2-\frac{\alpha N}{4}\right)\fc8{8-\alpha N}}.\notag
	\end{equation}
	Hence, the Laplacian of $u$ remains bounded if $\left\|\phi\right\|_{L^2}\le \left(\frac{1}{2C}\right)^{1/\alpha }$, which completes the proof of Theorem \ref{T2}.
\end{proof}
\section*{Acknowledgments}
This work is partially supported by the National Natural Science Foundation of China 11771389,  11931010 and 11621101.
\bibliographystyle{elsarticle-num}

\end{document}